\documentclass[a4paper,10pt]{article}
\usepackage{amsmath}
\usepackage{amssymb}
\usepackage{latexsym}
\usepackage{amsthm}
\usepackage{tikz}
\usetikzlibrary{decorations.pathmorphing}
\usepackage[normalem]{ulem}
\usepackage{mathrsfs}

\usepackage{calculator}

\newcommand{\comment}[1]{}

\newtheorem{thm}{Theorem}[section]
\newtheorem{cor}[thm]{Corollary}
\newtheorem{lem}[thm]{Lemma}

\newtheorem{conj}[thm]{Conjecture}

\theoremstyle{definition}
\newtheorem{defn}[thm]{Definition}
\newtheorem{ex}[thm]{Example}

\theoremstyle{remark}
\newtheorem{rem}{Remark}

\numberwithin{equation}{section}

\newcommand{\Z}{\mathbb{Z}}

\newcommand{\D}{\mathbb{D}}

\newcommand{\II}{\mathcal{I}}

\newcommand{\ol}{\overline}
\newcommand{\GL}{\mathcal{GL}}
\newcommand{\AL}{\mathcal{AL}}

\newcommand{\s}{\sigma}

\author{A.\ C.\ Burgess \footnotemark[1] 
\and 
P.\ Danziger \footnotemark[2] 
\and
T.\ Traetta \footnotemark[3]
}

\begin{document}
\date{\today}
\title{On the Oberwolfach problem for single-flip $2$-factors via graceful labelings}

\maketitle

\footnotetext[1]{Department of Mathematics and Statistics, University of New Brunswick, 100 Tucker Park Rd., Saint John, NB,  E2L 4L5, Canada}
\footnotetext[2]{Department of Mathematics, Ryerson University, 350 Victoria St., Toronto, ON,  M5B 2K3, Canada}
\footnotetext[3]{DICATAM, Universit\`{a} degli Studi di Brescia, Via Branze 43, 25123 Brescia, Italy}

\begin{abstract}
\noindent

Let $F$ be a $2$-regular graph of order $v$. 
The Oberwolfach problem $OP(F)$, posed in 1967 and still open,
asks for a decomposition of $K_v$ into copies of $F$.
In this paper we show that $OP(F)$  has a solution whenever
$F$ has a sufficiently large cycle which meets a given lower bound and, in addition, 
has a single-flip automorphism, which is an involutory automorphism
acting as a reflection on exactly one of the cycles of 
$F$. Furthermore, we prove analogous results for the minimum covering version and
the maximum packing version of the problem. 
We also show a similar result when the edges of $K_v$ have multiplicity 2, 
but in this case we do not require that $F$ be single-flip.

Our approach allows us to explicitly construct solutions to the Oberwolfach Problem with well-behaved automorphisms, in contrast 
with some recent asymptotic results, based on probabilistic methods, 
which are nonconstructive and do not provide a lower bound on the order of $F$ 
that guarantees the solvability of $OP(F)$.

Our constructions are based on a doubling construction which applies to 
graceful labelings of $2$-regular graphs with a vertex removed. We show that this class of 
graphs is graceful as long as the length of the path-component is sufficiently large.
A much better lower bound on the length of the path is given for an $\alpha$-labeling of such graphs to exist. 
\end{abstract}



\eject
\section{Introduction}
Given a graph $\Gamma$ with no loops, 
we denote by $V(\Gamma)$ and $E(\Gamma)$ the set of vertices and the multiset of edges of $\Gamma$, respectively, and say that $\Gamma$ is a graph of order $|V(\Gamma)|$ and size $|E(\Gamma)|$.
We denote by $\lambda K_v$ the \emph{$\lambda$-fold complete graph} on $v$ vertices, namely, 
the graph of order $v$ with every two distinct vertices joined by $\lambda$ edges. 
When $v$ is even, we denote by
$\lambda K_v-I$ (resp. $\lambda K_v+I$) the graph $\lambda K_v$ with the $1$-factor $I$ removed (resp. added). 
For brevity, letting $\s\in\{-1, 0, 1\}$ such that
$\lambda(v-1) \equiv \s\pmod{2}$, we set
\[
(\lambda K_v)^\s =
\begin{cases}
   \lambda K_v - I & 
   \text{if $\s = -1$},\\
   \lambda K_v  & \text{if $\s =0$}, \\  
   \lambda K_v + I & 
   \text{if $\s = 1$}. 
\end{cases}
\] 
We denote by
$P = u_0, u_1, \ldots, u_\ell$ the path of length $\ell\geq 0$ (an $\ell$-path)
whose edges are $\{u_i, u_{i+1}\}$ for $0 \leq i< \ell$. By joining the ends of $P$ when $\ell\geq 1$,
we obtain the cycle $(u_0, u_1, \ldots, u_\ell)$ of length $\ell+1$ (an $(\ell+1)$-cycle).
A $2$-cycle is then an edge counted twice.

A {\em factor} $F$ of $\Gamma$ is a spanning subgraph of $\Gamma$, namely, a subgraph of $\Gamma$ such that $V(F)=V(\Gamma)$; also, if $F$ is $i$-regular, we call $F$ an {\em $i$-factor}. 
In particular, a $1$-factor of $\Gamma$ (also called a {\em perfect matching}) is the vertex-disjoint union of edges of $\Gamma$ whose vertices partition $V(\Gamma)$, while a $2$-factor of $\Gamma$ is the vertex-disjoint union of cycles whose vertices span $V(\Gamma)$. Letting $2\leq \ell_1 <\ell_2< \ldots< \ell_t$, we denote by 
$[\;^{\alpha_1}\ell_1, ^{\alpha_2}\ell_2, \ldots, ^{\alpha_t}\ell_t]$ any $2$-regular graph containing 
$\alpha_i$ cycles of length $\ell_i$, $i=1,2, \ldots, t$. 


A {\em $2$-factorization} of a graph $\Gamma$ is a set $\mathcal{F}$ of $2$-factors of $\Gamma$ whose edge sets partition $E(\Gamma)$. 
Although it is well known that a regular graph has a 2-factorization if and only if every vertex has even degree, if we ask for the factorization $\mathcal{F}$ to contain only copies of a given $2$-factor $F$
(briefly, an $F$-factorization) then the problem becomes much harder. 
Indeed, if $\Gamma = \lambda K_v$ with $\lambda(v-1)$ even, and $F$ is a simple $2$-regular graph of order $v$, we have the Oberwolfach problem $OP(\lambda, F)$,
originally posed by Ringel in 1967 when $\lambda = 1$,
which is well known to be hard. The case $\lambda \geq 2$ was first studied in 1997 \cite{Gvo97}.

We are interested in the more general problem, 
denoted by $OP^\s(\lambda, F)$, which asks for an $F$-factorization of 
$(\lambda K_v)^\s$. By definition of $(\lambda K_v)^\s$, 
we have that $\sigma=0$ if and only if $\lambda(|V(F)|-1)$ is even, 
hence $OP^0(\lambda, F) = OP(\lambda, F)$. 
We refer to the extended version simply as the Oberwolfach problem and
write $OP^\s(F)$ when $\lambda=1$. 

When $\sigma=-1$, we also use the notation $OP^-(\lambda, F)$. In this case,
$\lambda(v-1)$ is odd and $(\lambda K_v)^\sigma = \lambda K_v - I$. This 
can be seen as the \emph{maximum packing version} of the original Oberwolfach problem, 
which was first considered in 1979 \cite{HuKoRo79} when $\lambda=1$, and has since received much attention.
Similarly, we use the notation $OP^+(\lambda, F)$ when $\sigma=1$. Again
$\lambda(v-1)$ is odd, but $(\lambda K_v)^\sigma = \lambda K_v + I$. This 
can be seen as the \emph{minimum covering version} of the original Oberwolfach problem, 
which  has been formally studied only recently in \cite{BBBSV19} when $\lambda=1$.

It is known that $OP^\s(\lambda, F)$, with $\s\in\{-1,0\}$, has no solution when
$\lambda=1$ and $F$ is isomorphic to 
$[^2 3]$, $[4, 5], [^2 3 , 5]$, or when 
$\lambda \not\equiv 0 \pmod 4$  and $F$ is isomorphic to $[^4 3]$ (see, \cite[Section VI.12]{Handbook}).
Otherwise, a solution is known 
when $F$ is uniform, that is, it contains cycles all of the same length
\cite{Alspach Haggvist 85, ASSW, Hoffman Schellenberg 91}, 
when the order of $F$ is less than $60$ 
\cite{DFWMR 10, SDTBD},
when $F$ is bipartite, and so contains only even cycles 
\cite{Alspach Haggvist 85, Bryant Danziger 11, Haggvist 85}, 
when the order of $F$ belongs to an infinite set of primes \cite{Bryant Schar 09},
when $F$ has order $2p$ and $p\equiv 5 \pmod{8}$ is a prime  \cite{AlBrHoMaSc 16},
or when $F$ has exactly two cycles 
\cite{BBR97, Buratti Traetta 12, Traetta 13}. 
Infinitely many other solutions are known, 
see \cite[Section VI.12]{Handbook} for a survey updated to 2006.
Although the original problem $OP(\lambda, F)$ and its maximum packing variant $OP^-(\lambda, F)$ are the most studied versions,
a complete solution to $OP(\lambda, F)$ or $OP^-(\lambda, F)$ has not yet been achieved.

As far as we know the only known results on $OP^+(\lambda, F)$ concern the case $\lambda=1$.  
It was shown in \cite{AMS87, LM93} that $OP^+([^n 3])$
has a solution if and only if $n$ is even and $n\geq 6$. 
The solvability of $OP^+(F)$ in the more general case in which $F$ is uniform of even order has been proven
in \cite{BBBSV19}, except possibly when $F=[^4 \ell]$ and $\ell\geq 5$ is odd. In the same paper the authors point out that the complete solution to $OP(F)$ when $F$ is bipartite, mentioned above, implies the solvability of $OP^+(F)$.

Infinitely many solutions to $OP(\lambda, F)$ and $OP^-(\lambda, F)$ 
have been constructed by requiring that they have an automorphism group
$G$ with an \emph{$f$-pyramidal} action on the vertex set, which means that $G$ fixes $f$ vertices and acts sharply transitively on the remaining. 
We speak of an $f$-pyramidal solution over $G$, although the most common terminology is \emph{regular} when $f=0$, and \emph{$1$-rotational} when $f=1$. 
Regular solutions have been built mainly over the cyclic group (see, for example, 
\cite{BuDel04, BuRaZu05,  Buratti Rinaldi 05}). 
It is worth pointing out that the version of the Oberwolfach problem concerning complete graphs of any infinite cardinality has recently been solved \cite{Costa20} by constructing regular solutions over any involution-free group of the same cardinality.
The $1$-rotational approach has proven much more successful (since 
it was formally introduced  in \cite{Buratti Rinaldi 08}) to construct solutions to $OP(F)$ for odd orders.
In \cite{BuRi09} the authors characterize completely the $1$-rotational solutions of $OP([3,2s])$ and via a composition technique infinitely many $1$--rotational solutions are built in \cite{RiTr11}.
Other $1$-rotational constructions include those given in \cite{BuZu98, BuZu01} concerning $OP([\,^{2n+1}3])$, 
those in \cite{Traetta 13} for the case where $F$ has two components and which rely on a doubling construction described in \cite{Buratti Traetta 12},  
and those given in a series of papers
\cite{OlPr03, Ol05, OlSt09, OlWil11} which mainly focus on the case where $F$ has three components.
It is worth pointing out that classical results on the Oberwolfach problem turn out to be $1$--rotational: for instance, the well-known Walecki construction yields a $1$--rotational solution to OP$([2n+1])$ (see, for example, \cite{Al08}). 
The structure of $2$-pyramidal solutions to $OP^-(F)$ has been formally investigated in \cite{BuTr15}.
There (see also \cite{HuKoRo79}) it is pointed out that any 1-rotational $F$-factorization of $K_{2n+1}$ easily yields a 2-pyramidal
$F^*$-factorization of $K_{2n+2}-I$ for a suitable $2$-factor $F^*$. However they prove that the converse does not hold in general, although it was shown to be true \cite{BuRiTr14} when $F^*$ is a cycle, namely, when we deal with $2$-pyramidal Hamiltonian cycle systems.
In this paper we also point out (see Theorem \ref{2pyrsols}) that 
every 1-rotational to $OP(F)$ of order $2n+1$ also gives rise to 
a regular solution to $OP^+(F^*)$ for a suitable $2$-factor $F^*$ of order $2n$.
A general treatment of $f$-pyramidal $2$-factorizations of $K_v$ can be found in \cite{BMR09}.

Most of the known solutions to $OP(F)$ 
concern factors $F$ with a specific cycle structure.
Here, we build on the techniques used in \cite{Traetta 13} when $F$ has two components, and deal with $OP^\s(\lambda, F$) when $F$ is a \emph{single-flip} $2$-regular graph, namely, it is simple and has an involutory automorphism $\varphi$ flipping exactly one cycle: this means that there is exactly one cycle of $F$ on which $\varphi$
acts as a reflection. Note that $2$-factors $F$ that are  bipartite, or uniform of odd order, or have two components of distinct parities (in which cases $OP(F)$ has proven to be solvable) fall into the class of single-flip $2$-factors. We also notice that the only known cases, mentioned above, where it is known that $OP^\s(\lambda, F$)
has no solution concern single-flip $2$-factors. The cycle structure of these $2$-regular graphs is the following.

\begin{thm}\label{reverse} 
$F$ is a single-flip $2$-regular graph if and only if
\begin{equation}\label{reverse:eq}
  F\simeq
  \begin{cases}
    [2\ell_0, 2\ell_1, \ldots, 2\ell_r, \;^{2}\ell_{r+1}, \ldots, \;^{2}\ell_{s}], 
    & \text{if $F$ has even order},\\
    [2\ell_0-1, 2\ell_1, \ldots, 2\ell_r, \;^{2}\ell_{r+1}, \ldots, \;^{2}\ell_{s}], 
    & \text{if $F$ has odd order},    
  \end{cases}  
\end{equation}
where $s\geq r\geq 0$,
$\ell_i\geq 2$ and $\ell_{j}\geq 3$ for every 
$0\leq i \leq r$ and $r+1\leq j\leq s$. 
\end{thm}
\begin{proof}
Let $F$ be a $2$-regular graph with an involutory automorphism $\varphi$ acting as a reflection on exactly one cycle of $F$, say $C_0$ of length
$2\ell_0-1$ or $2\ell_0$. This means that each other cycle $C$ of $F$ is 
is either fixed by $\varphi$, and this forces $C$ to be an even length cycle, or it is mapped onto another cycle of $F$ of the same length as $C$.
Therefore $F$ satisfies \eqref{reverse:eq}.

Conversely, if $F$ is a $2$-regular graph as in \eqref{reverse:eq}, we consider the involutory automorphism $\varphi$ of $F$ which rotates by $180^\circ$ each cycle of length $2\ell_1, \ldots, 2\ell_r$ cycle, swaps the pairs of cycles of length $\ell_{r+1}, \ldots, \ell_{s}$, 
and acts as a reflection on the cycle of length $2\ell_0$ or $2\ell_0-1$.
Therefore $F$ is a single-flip $2$-regular graph.
\end{proof}
\begin{rem}\label{rem:reverse}
Without loss of generality, in Theorem \ref{reverse}  we can assume that 
\begin{equation}\label{eq:reverse}
\ell_1, \ell_2,  \ldots, \ell_r\;\;\text{are pairwise distinct}.
\end{equation}
Otherwise, up to a permutation of the $\ell_i$s, we can have $\ell_1 = \ell_2$, 
hence
\[
[2\ell_1, \ldots, 2\ell_r, \;^{2}\ell_{r+1}, \ldots, \;^{2}\ell_{s}] = 
[2\ell_3, \ldots, 2\ell_r, \;^{2}(2\ell_1), \;^{2}\ell_{r+1}, \ldots, \;^{2}\ell_{s}].
\]
By induction, one can easily see that the assumption \eqref{eq:reverse} is not restrictive.
\end{rem}
We do not require that a solution to $OP^\s(\lambda, F)$ with $\lambda>1$ have distinct factors.
Therefore, if $F$ has odd order or $\lambda$ is even, then $\lambda$ copies of a solution to $OP(F)$ or $\lambda/2$ copies of a solution to $OP(2, F)$
yield a solution to  $OP(\lambda, F)$. 
This does not happen when $\lambda>1$ is odd and $F$ has even order:
in this case $\lambda$ copies of a solution $\mathcal{S}$  to $OP^\pm(F)$
provide an $F$-factorization of $\lambda(K_v\pm I)$, that is,
the graph $K_v\pm I$ with its edges counted $\lambda$ times, which is not isomorphic to
$\lambda K_v \pm I$. A solution to $OP^\pm(\lambda, F)$
can actually be obtained
by joining $\mathcal{S}$ with $(\lambda-1)/2$ copies of a solution of $OP(2, F)$.
This means that the cases $OP(2, F)$,
$OP(F)$ when $F$ has odd order, and $OP^{\pm}(F)$ when $F$ has even order, are crucial
and equally important 
in order to solve $OP^\s(\lambda, F)$ for every $\lambda$.

One of the main results of this paper is the following and concerns the existence of a solution to
$OP^\s(F)$ when $F$ is a single-flip $2$-regular graph with a sufficiently large cycle.

\begin{thm}\label{mainOP} 
Let $F=[h, 2\ell_1, \ldots, 2\ell_r, \;^{2}\ell_{r+1}, \ldots, \;^{2}\ell_{s}]$ 
where
$2\leq \ell_1< \ell_2< \ldots < \ell_r$ and 
$h,\ell_{r+1},\ell_{r+2}, \ldots, \ell_{s} \geq 3$,
and let $I=\{i\mid 2<\ell_i\; \text{is even}, 1\leq i\leq s\}$.
Then $OP^\s(F)$ has a $(1-\s)$-pyramidal solution whenever 
\[h> 16\max(1,h_0) + 20\max(3,h_1) + 29,\] where
 \begin{align*}
   & h_0 = 2|I|\max\{\ell_i + 3\mid i\in I\}  - 1,\\
   & h_1 = 
   7^{s-|I|-1}\max\{2\ell_i+1\mid i\not\in I\;\text{and}\; \ell_i\neq 2\}. 
 \end{align*}
\end{thm}
An improvement of this bound, when $F$ has additional properties, is given in 
Theorems \ref{mainOPeven} and \ref{mainOPodd}.
Furthermore, we prove the following analogous result concerning  $OP(2, F)$,
but in this case $F$ need not be a single-flip $2$-factor.
\begin{thm}\label{mainOP2}
Let $F=[h, \ell_1,\ldots, \ell_r, \ell'_1,\ldots, \ell'_s,]$ be any simple $2$-regular graph,
where the $\ell_i$s are even and the $\ell'_j$s are odd. 
Also, set $\ell=\max\{\ell_1,\ldots, \ell_r\}$ and  $\ell'=\max\{\ell'_1,\ldots, \ell'_s\}$.
Then $OP(2, F)$ has a $1$-rotational solution whenever
\[h> 
  \begin{cases}
    12r(\ell+3) + 7^{s}(2\ell'+1) -6  &\text{if $r,s > 0$}, \\    
    2(r+1)(\ell+1) +5 &\text{if $r> 0=s$},  \\    
    3\cdot7^{s-1}(2\ell'+1) &\text{if $s> 0=r$}. \\       
  \end{cases} 
\]   
\end{thm}

We would like to point out that recently probabilistic methods have been used 
to provide asymptotic results on longstanding open problems in design theory and graph decompositions. 
In particular, in \cite{GJKKO} the authors 
provide an asymptotic result on the solvability of $OP(F)$. 
However, the results in \cite{GJKKO} are not constructive 
and do not provide a lower bound on the order of $F$ that guarantees the solvability of $OP(F)$.
In strong contrast, 
our approach allows us to explicitly construct a solution to $OP^\s(F)$ whenever $F$ satisfies the assumptions of Theorem \ref{mainOP} which provides an explicit lower bound. Furthermore, our solutions have a well-behaved automorphism group.

Given a graph $\Gamma$ with $\epsilon$ edges, a \emph{graceful labeling} of $\Gamma$ is a labeling of its vertices with integers from 0 to $\epsilon$, so that the differences on the edges give every integer from 1 to $\epsilon$.  
A formal definition of graceful labelings
of a $2$-regular graph with a vertex removed, 
that is, the vertex disjoint union of cycles and a path, appears in Definition \ref{def:GL}.
The proof of Theorem \ref{mainOP} relies on a doubling construction defined in
\cite{Buratti Traetta 12} which applies to graceful labelings of these graphs.
This approach was used  in \cite{Traetta 13} by the third author to solve completely 
$OP(\lambda, F)$ when $F$ consists of exactly two components. 
Much of this article is devoted to constructing such graceful labelings which we prove to exist whenever the path is sufficiently long. 
More precisely, we prove the following.

\begin{thm}\label{mainGL} Let $\Gamma$ be the graph whose components are $s$ cycles of length $\ell_1, \ell_2,\ldots, \ell_s\geq 3$ and a path of length $m$,
and let $I=\{i\mid \ell_i\; \text{is even}, 1\leq i\leq s\}$.
Then $\Gamma$ has a graceful labeling whenever
\[m\geq 6\max(1,m_0) + 7\max(3,m_1) + 9,\] 
where 
  \begin{align*}
   & m_0 = 2|I|\max\{\ell_i + 3\mid i\in I\}  - 1,\\
   & m_1 = 
   7^{s-|I|-1}\max\{2\ell_i+1\mid i\not\in I\}. 
 \end{align*}
 \end{thm}

The graphs considered in Theorem \ref{mainGL} belong to the class $\mathcal{C}$ of graphs whose size equals the order minus one
(which is the minimum size for a graph to possibly have a graceful labeling).
The gracefulness (namely, the existence of a graceful labeling) of the graphs in  $\mathcal{C}$ has been widely studied, 
mainly focusing on trees for which Ringel and Kotzig posed the well-known 
graceful tree conjecture which, despite the efforts of many, remains open (see \cite{Ga}). 

The graphs for which we are interested in building graceful labelings are exactly those in 
$\mathcal{C}$ of maximum degree 2.
Their structure can be described in terms of the size of their components, 
which are cycles and exactly one path.
Pictorially, these graphs resemble a $1$ followed by a number of $0$s,
which represent the path and the cycle components, respectively. 
For this reason, 
we refer to them as \emph{zillion} graphs 
and denote the class of all such graphs by $\mathcal{Z}$.
Despite having an easy-to-describe structure, their gracefulness is still uncertain. 
Graceful labelings of zillion graphs with two components were built in \cite{Traetta 13} settling a conjecture posed by Frucht and Salinas  \cite{FS85} in 1985. 
As far as we know,
Theorem \ref{mainGL} represents the first result on the gracefulness of graphs in $\mathcal{Z}$ with more than two components. This result led us to make the following conjecture.


\begin{conj}\label{conjectureweak} There exists a polynomial $f\in\Z[x,y]$ 
such that
any zillion graph with $s$ cycles of maximum length $r$ and a path of length 
$\geq f(r, s)$ is graceful.
\end{conj}
In fact, we actually believe that $f$ is quadratic. 
Theorem \ref{mainGL} proves this conjecture 
with a quadratic $f$
when the graph has only even cycles.
In this case we will give a slightly better bound, $f(r, s)$, in Corollary \ref{cor:main1}.

We actually believe that the following stronger version holds, 
where the function $f$ does not depend on the cycle lengths.
\begin{conj}\label{conjecturestrong} There exists a quadratic polynomial $f\in\Z[x]$ such that
any zillion graph with $s$ cycles and a path of length $\geq f(s)$ 
is graceful. Furthermore, $f$ is linear for bipartite zillion graphs.
\end{conj}

In \cite[Theorem 9]{Ko84} Kotzig gave a lower bound on the number of vertices for a 2-regular graph plus an isolated vertex to be graceful in terms of the number of odd cycles. Specifically, he showed that $|E(\Gamma)| \geq \omega(\omega+2)$, where $\omega$ is the number of odd cycles. 
By adapting Kotzig's proof, one can
show that if a zillion graph $\Gamma$ contains $\omega\leq s$ odd length cycles, 
then for $\Gamma$ to be graceful we must have $|E(\Gamma)|> \omega(\omega+2)$.
In light of this, we propose the following conjecture on the size of $\Gamma$.

\begin{conj}\label{conjecturestrongest} 
There exist a quadratic polynomial $f\in\Z[x]$ and a linear polynomial $g\in\Z[x]$ 
such that
any zillion graph $\Gamma$ with $\omega$ odd cycles and $\mu$ even cycles 
is graceful whenever $|E(\Gamma)|\geq f(w) + g(\mu)$.
\end{conj}



This paper is structured as follows. In Section 2 we recall the basic concepts of a graceful labeling and an $\alpha$-labeling which are both slightly generalized, 
and some related results. We also recall the concept of a starter which we use to construct pyramidal solutions to $OP^\s(\lambda, F)$ and explain in detail how we use a doubling construction defined in \cite{Buratti Traetta 12}. 
In Section 3 we provide some extension lemmas for graceful and $\alpha$-labelings which will be used in 
Sections 4 and 5 where we build $\alpha$-labelings and graceful labelings, respectively, for graphs in $\mathcal{C}_2$ whose path-component is
sufficiently long (see Theorem \ref{mainGL}). 
New results on the solvability of  $OP^\s(\lambda, F)$ are then given in Section 6 (see Theorem \ref{mainOP}). The paper ends with some concluding remarks.

\section{Preliminaries}
Given two integers $a$ and $b$, we denote by $\II(a, b)$ the interval of all integers $x$ such that $a \leq x \leq b$. Of course, if $b<a$, then $\II(a, b)$ will be the empty set. Finally, for $a \in \Z$ and 
$X\subseteq \Z$, we set $aX=\{ax\mid x \in X\}$.

Let $\D$ be an integral domain and let $\Gamma$ be a simple graph with vertices in $\D\ \cup\ \{\infty\}$ where 
$\infty\not\in \D$. 
As usual, the {\em list of differences of $\Gamma$} is the multiset $\Delta \Gamma$ of all differences 
$x-y$ with $(x,y)$ an ordered pair of adjacent vertices of $\Gamma$ and $x\neq\infty\neq y$. 
Also, given $a,b\in \D$, with $a\neq 0$,
we can define the graph $a\Gamma + b$ obtained from $\Gamma$ by replacing each vertex $x\neq\infty$ with $ax+b$. This new graph is isomorphic to $\Gamma$ and satisfies the following properties:
\begin{align*}
  V(a\Gamma+b) = aV(\Gamma)+b \;\;\text{and} \;\; \Delta(a\Gamma+b)& =a \Delta \Gamma;
\end{align*}
We point out that the integral domains considered in this paper will be either the ring of integers $\Z$ or the ring $\Z_n$ of integers modulo $n$.

\subsection{Labelings of zillion graphs}
In this section we recall the basic concept of labeling and some well known results on graceful labelings and 
$\alpha$-labelings of a zillion graph, which we recall is any $2$-regular graph with a vertex removed. Since the components of such a graph are cycles and a 
path, its structure is determined by the lengths of its components which form a multiset.
Therefore, to avoid confusion, we introduce here the notation on multisets which we use throughout the paper.

We use the left exponential notation and denote by 
$L=\{^{a_1}\ell_1, \;^{a_2}\ell_2,\; \ldots, \;^{a_t}\ell_t\}$ the multiset containing 
$a_i$ copies of the integer $\ell_i$, for $1\leq i\leq t$; also, we denote by $|L| = \sum_{i=1}^t a_i$ the cardinality of $L$, and by $\max(L)$ and $\min(L)$ the maximum and the minimum element of $L$, respectively, unless $L$ is empty in which case we set $\max(L) = 0 = \min(L)$.
Furthermore, for every $b\in\mathbb{N}$ and $q, r\in \mathbb{Q}$, we let
\begin{itemize}
  \item[$-$] $^b L = \{^{a_1b}\ell_1, \;^{a_2b}\ell_2,\; \ldots, \;^{a_tb}\ell_t\}$;
  \item[$-$] $qL + r = \{^{a_1}(q\ell_1+r), \;^{a_2}(q\ell_2+r),\; \ldots, \;^{a_t}(q\ell_t+r)\}$;
\end{itemize}
When $L = \;\{^a \ell\}$, we simply write $^a \ell$ in place of $L$, 
and if $a=1$, we write $\ell$ in place of $\;^a\ell$.

Given a non-negative integer $m$ and a list 
$L_1,  L_2, \ldots, L_t$, where each $L_i$ is a  multiset of integers $\geq 2$,
we denote by 
\begin{itemize}
  \item[$-$] $[L_1, L_2, \ldots, L_t]$ the graph containing
an $\ell$-cycle for every $\ell~\in~\bigcup_{i=1}^t L_i$;
  \item[$-$] $[L_1, L_2, \ldots, L_t\mid m]$  the vertex-disjoint union of 
  the graph $[L_1,$ $L_2, \ldots, L_t]$ and a path with $m\geq 0$ edges (when $m=0$, this path is an isolated vertex).
\end{itemize}
Clearly, $[L_1, L_2, \ldots, L_t\mid m]$ can be obtained from the $2$-regular graph 
$[L_1, L_2, \ldots,$ $L_t,  m+2]$ by removing a vertex from the $(m+2)$-cycle.

We write $[L_1, L_2, \ldots, L_t]_\epsilon$ (resp.\ $[L_1, L_2, \ldots, L_t\mid m]_\epsilon$) whenever we need to
specify the size $\epsilon$ of the related graph; clearly, $\epsilon$ is the sum of all cycle and path lengths, hence
$\epsilon$ (resp. $\epsilon$+1) is the order of $[L_1, L_2, \ldots, L_t]_\epsilon$ (resp.\ $[L_1, L_2, \ldots, L_t\mid m]_\epsilon$).
Note that $[\ell]$ denotes 
a cycle of length $\ell$ (an \emph{$\ell$-cycle}), while by $[\;-\mid m]$ we
denote a path with $m$-edges (an \emph{$m$-path}).\\

In this paper, we only deal with labelings of 
$[L \mid m]$ over $\Z$, where $L$ is a multiset of integers $\geq2$.
A \emph{labeling} of $[L \mid m]$ is simply a graph $\Gamma$ 
isomorphic to  $[L \mid m]$ with vertices in $\Z$.

As shown in \cite[Theorem 3.2]{Buratti Traetta 12} (with a slightly different terminology)
labelings over $\Z$ can be used to construct solutions to $OP(2,F)$. More precisely, we have the following.

\begin{thm}\label{1rotsols2}
If there is a labeling $\Gamma$ of $[\ell_1, \ell_2, \ldots, \ell_t \mid \ell^*]_{n-1}$ such that\
\begin{equation}\label{1rotsols2:cond}
  \text{$V(\Gamma) = \{0,1, \ldots, n-1\}$,\; and\; $\Delta \Gamma = \;^2\{1, 2, \ldots, n-1\}\pmod{n}$,}
\end{equation}
then there exists a $1$-rotational solution to $OP(2, [\ell^*+2, \ell_1, \ell_2, \ldots, \ell_t])$.
\end{thm}
\begin{proof} 
Let $\Gamma$ be a labeling of $[\ell_1, \ell_2, \ldots, \ell_t \mid \ell^*]_{n-1}$ satisfying \eqref{1rotsols2:cond}. By joining the ends of the path in $\Gamma$ with $\infty$, we obtain
a $2$-regular graph $\Gamma^*$ isomorphic to $[\ell^*+2, \ell_1, \ell_2, \ldots, \ell_t]$.
If we consider the vertices of $\Gamma^*$ modulo $n$, 
then $V(\Gamma^*) = \Z_{n} \ \cup\ \{\infty\}$, and by \eqref{1rotsols2:cond}
it follows that $\Delta \Gamma^* = \;^2(\Z_{n}\setminus\{0\}$). This guarantees that  
the list $\{\Gamma + i\mid  i\in \Z_n\}$ of all translates of 
$\Gamma$ is a solution to $OP(2, [\ell^*+2, \ell_1, \ell_2, \ldots, \ell_t])$, which is 
$1$-rotational by construction.
\end{proof}

Two important classes of labelings satisfying condition \eqref{1rotsols2:cond} are represented by
graceful labelings and $\alpha$-labelings, both introduced by A. Rosa in \cite{Ro66}
(see \cite{Ga} for a dynamic survey on the topic). 
In order to develop a method which allows us to construct graceful or $\alpha$-labelings of $[L_0, L_1\mid \mu]$
starting from suitable labelings of $[L_0\mid m_0]$ and $[L_1\mid m_1]$, in the following two sections
we slightly generalize the two concepts of graceful and $\alpha$-labelings. 
%

\subsubsection{Graceful labelings}\label{subsec:beta}
Using our terminology, we generalize the concept of a graceful labeling of $[L\mid m]$ to include the
case where $[L\mid m]$ contains exactly one $2$-cycle.

\begin{defn}\label{def:GL}
Let $J$ be an interval of $\Z$ of length $\epsilon+1$ and let $x,y\in J$.
Also, we assume that $[L\mid m]_\epsilon$ contains at most one $2$-cycle.

A \emph{graceful labeling} $(\GL)$ of $[L \mid m]_\epsilon$ with \emph{parameter set} $(J, x, y)$
-- briefly, a $\GL(J, x, y)$ of $[L \mid m]_\epsilon$ -- 
is a labeling $\Gamma$ of $[L \mid m]_\epsilon$
which satisfies the following three conditions:
\begin{enumerate}
  \item $V(\Gamma)= J$;
  \item $x$ and $y$ are the end-vertices of the $m$-path in $\Gamma$;
  \item 
  $  
  \Delta \Gamma = 
  \begin{cases} 
    \pm \II(1, \epsilon) & \text{if $[L \mid m]$ has no $2$-cycle},\\
    \pm \{1, \;^23\}\ \cup\ \pm\II(4, \epsilon) & \text{if $[L \mid m]$ has exactly one $2$-cycle}.  
  \end{cases}
  $
\end{enumerate}
\end{defn}


If the vertex set $J=\II(0,\epsilon)$, we will often drop the parameter $J$ and simply speak of a graceful labeling or a $\GL(x,y)$ of $[L \mid m]$. Note that when $J=\II(0,\epsilon)$ and $[L\mid m]$ has no 
$2$-cycle, we obtain the classic definition of a graceful labeling (see \cite{Ga}). 

\begin{rem}\label{rem:GL1}
  If $[L\mid m ]$ contains exactly one $2$-cycle, 
  the choice of missing differences $\pm2$ and having differences $\pm3$ appear twice 
  in a graceful labeling $\Gamma$ of $[L\mid m ]$ is somewhat arbitrary, 
  but this choice will prove convenient later on.
  Further,
  the $2$-cycle $C$ in $\Gamma$ must have the form $C = (a, a+3)$ for some $a\in\Z$. Hence 
  the differences $\pm 3$ (with multiplicity $2$) appear only in $C$.
\end{rem}

\begin{rem}\label{rem:GL2} Given a $\GL(J, x, y)$ of $[L \mid m]$, 
the translate $\Gamma + z$ of $\Gamma$ by $z\in\Z$
is clearly a $\GL(J+z, x+z, y+z)$ of $[L \mid m]$. Therefore, if 
$J = \II(u, u + \epsilon)$, then the graph $\Gamma - u$ 
is a $\GL(x - u, y - u)$ of $[L \mid m]$.
\end{rem}

\subsubsection{$\alpha$-labelings}\label{subsec:alpha}
The graph $[L \mid m]$ is bipartite if and only if all its  cycles have even length. 
In this case, the original concept of an $\alpha$-labeling of $[L \mid m]$ can be generalized as follows.
\begin{defn}\label{def:AL} 
Let $J_i = \II(w_i, z_i)$, for $i=1, 2$, such that 
\[ z_1 < w_2\;\;\; \text{and}\;\;\; |J_1| - |J_2|\in\{-1,0,1\}.\] 
Also, let $x,y\in J_1\cup J_2$ and set $\epsilon = |J_1| + |J_2| - 1$.

An \emph{$\alpha$-labeling} $(\AL)$ of a bipartite graph $[L \mid m]_\epsilon$ 
with \emph{parameter set} $(J_1, J_2, x, y)$ -- briefly, an $\AL(J_1, J_2, x, y)$ of $[L \mid m]_\epsilon$ -- 
is a labeling $\Gamma$ of  $[L \mid m]$ 
which satisfies the following three conditions:
\begin{enumerate}
  \item the partite sets of $\Gamma$ are $J_1$ and $J_2$;
  \item $x$ and $y$ are the end-vertices of the $m$-path in $\Gamma$;
  \item $\Delta \Gamma = \pm \II\big(w_2 - z_1, z_2 - w_1\big)$.
\end{enumerate}
\end{defn}
Note that by condition \ref{def:AL}.(3), the $\alpha$-labeling $\Gamma$ has no repeated differences, hence 
$[L\mid m]$ has no repeated edges, that is, $[L\mid m]$ has no $2$-cycle.

When $w_2=z_1+1$, namely, when the two intervals $J_1$ and $J_2$ are consecutive, then the $\alpha$-labeling $\Gamma$ is also a graceful labeling.
Furthermore, if $w_1=0$ and $w_2=z_1+1$, that is, $J_1\ \cup\ J_2 = \II(0, \epsilon)$, 
we obtain the classic definition of an $\alpha$-labeling (see \cite{Ga}).  
Whenever we do not specify the intervals $J_1$ and $J_2$ and simply speak of an 
$\alpha$-labeling or an $\AL(x,y)$ of $[L \mid m]_\epsilon$, 
it will be understood that
\[
  J_1 = \II\left(0, \left\lfloor \frac{\epsilon - 1}{2} \right\rfloor \right)\;\;\; \text{and}\;\;\; 
  J_2 = \II\left(\left\lfloor \frac{\epsilon + 1}{2} \right\rfloor, \epsilon\right).
\]  

Let $\Gamma$ be as in Definition \ref{def:AL}. Given two integers  $a_1$ and $a_2$ such that 
\[ \text{either}\;\; a_1 < a_2 + w_2 - z_1\;\; \text{or}\;\; a_1>a_2 + z_2 - w_1.\]
we define $\Gamma + (a_1,a_2)$ to be the graph obtained from $\Gamma$ by replacing each vertex 
$u\in J_i$ with $u+a_i$, for $i=1,2$. Clearly, $\Gamma + (a_1,a_2)$ is isomorphic to $\Gamma$, and
it is not difficult to check that:
\begin{enumerate}
  \item[$a.$] $\Gamma + (a_1,a_2)$ is a bipartite graph whose parts are $J_i + a_i$, for $i=1,2$;
  \item[$b.$] $\Delta \big(\Gamma + (a_1,a_2)\big) = 
    \pm \II(w_2 - z_1 + a_2 - a_1, z_2 - w_1 + a_2-a_1)$;
\end{enumerate}
Hence, $\Gamma+(a_1, a_2)$ is an $\AL(J_1 + a_1, J_2 + a_2, x', y')$ of $[L \mid m]$, for suitable integers 
$x'$ and $y'$. When $a_1=a_2$, then $\Gamma+(a_1, a_2) = \Gamma+a_1$.

\begin{rem}\label{rem: AL} If $\Gamma$ is an $\alpha$-labeling of $[L\mid m]_\epsilon$ with parameter set 
$(J_1, J_2, x, y)$ (Definition \ref{def:AL}), then 
$\Gamma' = \Gamma + (-w_1, z_1-w_1-w_2+1)$ is an 
$\AL(J'_1, J'_2, x', y')$ of $[L\mid m]_{\epsilon}$ for suitable integers $x'$ and $y'$, where
\[
J'_1 = J_1 - w_1 = \II(0, z_1-w_1),\; J'_2 = J_2 + (z_1-w_1-w_2+1) = \II(z_1-w_1 + 1, \epsilon).
\] 
\end{rem}

\subsubsection{Known results on graceful and $\alpha$-labelings of $[L \mid m]$}

There are many results on the existence of graceful labelings of 
$[L \mid m]$, where $L$ is a list of integers greater than $2$ (see \cite{Ga}). 
Just to mention a few, we recall the following.
\begin{enumerate}
\item If $[\ell_1, \ell_2, \ldots, \ell_t\mid 0]$ has a graceful labeling, then 
      $\sum_{i=1}^t \ell_i \equiv 0$ or $3 \pmod{4}$~\cite{Ro66}.
\item The above necessary condition is also sufficient when $t=1$ \cite{Ro66} or $t=2$~\cite{AbKo96}.
\item If $m\geq 1$, then $[\ell\mid m]$ has a graceful labeling if and only if $\ell + m \geq 6$ \cite{Traetta 13}.
\end{enumerate}

The following lemma generalizes a result by Abrham \cite{Ab93} concerning the
existence of an $\alpha$-labeling of a path with a given end-vertex.

\begin{lem}\label{paths} Let $J_i = \II(w_i, w_i+\gamma_i)$ be a nonempty interval, for $i=1,2$, and assume that 
$w_1+\gamma_1<w_2$ and $|\gamma_1 -\gamma_2|\leq 1$.
Then there is an $\AL(J_1, J_2, x, y)$ of $[\;-\mid \gamma_1 + \gamma_2 - 1]$ in each of the following cases:
\begin{enumerate}
\item $(x, y)=(w_1+i, w_1+\gamma_1-i)$ 
if $\gamma_1 = \gamma_2+1$, and $i\in\II(0,\gamma_1)\setminus\{\gamma_1/2\}$;
\item $(x, y)=(w_1+i, w_2+i)$  if $\gamma_1= \gamma_2$, and $i\in\II(0,\gamma_1)$;
\item $(x, y)=(w_2+i, w_2+\gamma_2-i)$ if $\gamma_1 = \gamma_2-1$, and $i\in\II(0,\gamma_2)\setminus\{\gamma_2/2\}$.
\end{enumerate}
\end{lem}
\begin{proof} When $w_1=0$ and $w_2 = \gamma_1+1$, the existence of an $\alpha-$labeling, say $\Gamma_0$,
of $[-\mid m]$ satisfying the assertion is proven in \cite{Ab93}. In the general case, 
the graph $\Gamma_0 + (w_1, w_2 - \gamma_1 - 1)$ provides the desired $\alpha$-labeling.
\end{proof}

\subsection{Basics on $2$-starters of $\Z_{2n}$}
The concept of a  \emph{$2$-starter} (and more generally of a \emph{$k$-starter}) over an arbitrary group $G$ was formally introduced in \cite{Buratti Rinaldi 08} to characterize the solutions to $OP$ that are  $1$-rotational over $G$. 
They show, among other things, that a necessary condition for the existence of a $1$-rotational solution to $OP(F)$ over a group $G$ of even order
is that all involutions of $G$ (i.e., the elements of order $2$) are pairwise conjugate. Indeed, it is shown that many groups with this property give rise to $1$--rotational solutions enabling the authors to conjecture that such a necessary condition is also sufficient. A wider class of groups supporting the conjecture can be found in \cite{Tr10}.

In this paper we only deal with $2$-starters over $\Z_{2n}$ whose definition is the following.

\begin{defn}\label{2-starter}
  A \emph{$2$-starter} (over $\Z_{2n}$) is a $2$-regular graph $\Sigma$
  with vertex-set $V(\Sigma) = \Z_{2n}\ \cup\ \{\infty\}$ such that
  \begin{enumerate}
    \item\label{2-starter: item 1} $\Delta \Sigma \supset \Z_{2n}\setminus\{0\}$, and
    \item\label{2-starter: item 2} $\Sigma + n = \Sigma$.
  \end{enumerate}
If $\Sigma \simeq [\ell_\infty, \ell_1, \ldots, \ell_t]$ and $\ell_\infty$ is the length of the cycle through $\infty$, 
we speak of an $\left[\underline{\ell_\infty}, \ell_1, \ldots, \ell_t\right]$-starter.
\end{defn}

\begin{rem}\label{starter structure}
By condition \ref{2-starter}.(\ref{2-starter: item 2}), the translation by $n$ fixing $\infty$, say $\tau$, is clearly an involutory automorphism of $\Sigma$.
The cycle structure of $\Sigma$, induced by the action of $\tau$, is described in \cite{Buratti Rinaldi 08}. 
In particular,  $\Sigma$ contains exactly one cycle on which $\tau$ acts as a reflection, 
the cycle through $\infty$, whose length $\ell_\infty$ is necessarily odd, 
since $\tau$ fixes no vertex other than $\infty$.
Therefore $\Sigma$ is a single-flip $2$-regular graph.
\end{rem}
The following theorem, special
case of a more general result proved in \cite[Theorem 2.3]{Buratti Rinaldi 08}, shows how a $2$-starter can generate a solution to $OP$.

\begin{thm}\label{1rotsols}
If $\Sigma$ is a $2$-starter of $\Z_{2n}$, 
then $\mathcal{S} = \{\Sigma + i\mid 0\leq i\leq n-1\}$ is a $1$-rotational solution to OP$(\Sigma)$.
\end{thm}

We point out that a solution $\mathcal{S}$ to OP constructed as in Theorem \ref{1rotsols}
is \emph{$1$-rotational} since the group of translations of $\Z_{2n}$, fixing $\infty$, 
is an automorphism group of $\mathcal{S}$ which fixes one vertex and acts sharply transitively on the
remaining.

A $2$-starter $\Sigma$ of $\Z_{2n}$ can be easily modified to generate 2-factorizations of $K_{2n+2} - I$ and $K_{2n} + I$. More precisely, we have the following.

\begin{thm}[]\label{2pyrsols}
If $\Sigma$ is an $\left[\underline{\ell_{\infty}}, \ell_1, \ldots, \ell_t\right]$-starter of $\Z_{2n}$, then 
\begin{enumerate}
  \item\label{2pyrsols:1} there is a $2$-pyramidal solution to $OP^-([\ell_{\infty}+1, \ell_1, \ldots, \ell_t])$,
  \item there is a regular solution to $OP^+([\ell_{\infty}-1, \ell_1, \ldots, \ell_t])$.
\end{enumerate}
\end{thm}
\begin{proof}
  Let $\mathcal{S} = \{\Sigma+i\mid 0\leq i\leq n-1\}$ be the $1$-rotational solution to
  $OP(\left[{\ell_{\infty}}, \ell_1, \ldots, \ell_t\right])$ built in Theorem \ref{1rotsols}.
  By remark \ref{starter structure},
  the cycle of $\Sigma$ passing through $\infty$, say $C_\infty$,
  has odd length $\ell_{\infty} =2h+1$ and 
  $C_{\infty} +n = C_{\infty}$, therefore
  $C_{\infty} = (x_1, \ldots, x_h, \infty, x_h+n, \ldots, x_1 + n)$. 
  From $\Sigma$ we construct two $2$-regular graphs $\Sigma^-$ and $\Sigma^+$ 
  by only modifying the cycle $C_\infty$:
  \begin{enumerate}
    \item we obtain $\Sigma^-$ by replacing the edge
    $\{x_1, x_1+n\}$ with the $2$-path $P=x_1, \infty', x_1+n$,  
    \item we obtain $\Sigma^+$  by removing $\infty$ and joining $x_h$ and
    $x_h+n$.
  \end{enumerate}
  We claim that $\mathcal{S}^* = \{\Sigma^*+i\mid 0\leq i\leq n-1\}$ with $*\in\{+, -\}$ are
  the desired solutions to OP.
  
  Note that $\Sigma^-\simeq [\ell_{\infty}+1, \ell_1, \ldots, \ell_t]$ has order $2n+2$, and
  contains $\{\infty', x_1\}$, $\{\infty', x_1+n\}$ and all the edges of 
  $\Sigma$ except for $\{x_1, x_1+n\}$. Therefore $\mathcal{S}^-$ is a 
  $\Sigma^-$-factorization of $K_{2n+2} - I$ where 
  $V(K_{2n+2}) =\Z_{2n}\ \cup\ \{\infty, \infty'\}$
  and $I =\{\{x_1 +i, x_1+n+i\}\mid 0\leq i\leq n-1\}\ \cup \ \{\infty, \infty'\}$, which is clearly
  a $1$-factor.
  
  Similarly $\Sigma^+\simeq [\ell_{\infty}-1, \ell_1, \ldots, \ell_t]$ has order $2n$, and
  contains $\{x_h, x_h+n\}$ and all the edges of 
  $\Sigma$ except for $\{\infty, x_h\}$, $\{\infty', x_h+n\}$. 
  Therefore $\mathcal{S}^+$ is a 
  $\Sigma^+$-factorization of $K_{2n} + J$ where 
  $V(K_{2n}) =\Z_{2n}$
  and $J =\{\{x_h +i, x_h+n+i\}\mid 0\leq i\leq n-1\}$, which is clearly a $1$-factor.  
  
  Considering that $\Sigma^* + n = \Sigma^*$ for $*\in\{+, -\}$, the group of translations of
  $\Z_{2n}$ is 
  a 2-pyramidal automorphism group of $\mathcal{S}^-$ and 
  a regular automorphism group of $\mathcal{S}^+$, and this completes the proof.
\end{proof}
We point out that Theorem \ref{2pyrsols}.(\ref{2pyrsols:1}) was proven in \cite[Lemma 2.4]{HuKoRo79}.

A powerful method to construct $2$-starters, and hence pyramidal solutions to $OP$, 
consists in applying a 
doubling construction, described in \cite[Theorem 3.4]{Buratti Traetta 12}, to a suitable labeling of
$[L \mid m]$. 
The following result, proven in \cite[Proposition 6.4]{Buratti Traetta 12}, 
shows how this method works when applied to classic graceful labelings. We provide a brief proof which illustrates the construction with respect to our notation.

\begin{thm}
\label{doubling1} Let $L$ be a list of integers greater than 2. 
  If there exists a graceful labeling of $[L \mid m]$, then 
  there is a $\left[\underline{2m+3}, \;^2 L\right]$-starter. 
\end{thm}
\begin{proof}
  Let $\Gamma$ be a $\GL(x,y)$ of $[L \mid m]_{n-1}$ and set 
  $\Sigma = \Gamma\ \cup\ (\Gamma+n)\ \cup\ \mathcal{E}$, where 
  $\mathcal{E} =\{\{\infty, x\}, \{\infty, x+n\}, \{y, y+n\}\}$.
  Clearly, $\Gamma$ and $\Gamma+n$ are vertex-disjoint. 
  Also, the edges of $\mathcal{E}$ join the $m$-path in $\Gamma$ 
  (whose end vertices are $x$ and $y$) and its translate by $n$, belonging to $\Gamma+n$, 
  into a $(2m+3)$-cycle.
  Therefore, $\Sigma$ is isomorphic to $[2m+3, \;^2L]$. 
  
  Considering modulo $2n$ the vertices of $\Sigma$ different from $\infty$, 
  one can easily check that  $\Sigma$ is a 
  $\left[\underline{2m+3}, \;^2 L\right]$-starter of $\Z_{2n}$.
\end{proof}

The doubling construction defined in \cite{Buratti Traetta 12} has a wide application.
In the following we illustrate how the application of this construction can be generalized to 
graceful labelings of $[L\mid m]$ containing exactly one $2$-cycle.

\begin{thm}
\label{doubling2} Let $L$ be a list of integers greater than 2.
If there is a graceful labeling of $[L, 2 \mid m]_\epsilon$ whose path contains the difference $\epsilon$,
then there is a $\left[\underline{2m+7}, 4, \;^2 L\right]$-starter. 
\end{thm}
\begin{proof} Let $n=\epsilon+3$, and let $\Gamma = \Gamma_0\ \cup\ (u, u+3) \ \cup\ P$ be 
a graceful labeling of $[L, 2 \mid m]_{n-3}$,
where $P$ denotes the $m$-path of $\Gamma$. In the following, we consider the vertices of $\Gamma$ 
modulo $2n$.

We are going to construct a $4$-cycle $C$ and a $(2m+7)$-cycle $C_{\infty}$ so that
the union of $C_{\infty}$, $C$, $\Gamma_0$, and $\Gamma_0+n$ is the desired 
$[2m+7, 4, \;^2 L]$-starter of $\Z_{2n}$.

First, let $C = (u, u + 3, u + n, u + 3 + n)$. Clearly,
\begin{equation}\label{C4}
  C + n = C \;\;\text{and}\;\; \Delta C \supseteq \pm\{3, n-3\}
\end{equation}

We proceed by constructing the $(2m+7)$-cycle $C_{\infty}$.
Since by assumption $\epsilon\in \Delta P$, the edge $\{0, n-3\}$ belongs to $P$, that is,
\[
  \text{$P = w_0, \ldots, w_{i}, w_{i+1}, \ldots, w_m$ with  $w_{i}=0$ and $w_{i+1} = n-3 = \epsilon$},
\]
for some $i\in\{0,\ldots, m-1\}$. Let 
$\mathcal{E} =\{\{\infty, w_0\}, \{\infty, w_0+n\}, \{w_m, w_m+n\}\}$, and let $Q$ denote 
the following $(m+2)$-path ,
\[
  Q = w_0, \ldots, w_{i}, n-2, -1,  w_{i+1}+n, w_{i+2}+n, \ldots, w_m+n.
\]
\begin{center}
\begin{tikzpicture}[x=1cm,y=1cm,scale=0.75]
\draw (2.5,2) circle (3pt)[fill] node[above=2pt] {$w_0$};
\draw (2,2) -- (2.5,2);
\draw (2.5,0) circle (3pt)[fill] node[below=2pt] {$w_m$};
\draw (2,0) -- (2.5,0);
\draw (11.5,2) circle (3pt)[fill] node[above=2pt] {$w_0$};
\draw (11,2) -- (11.5,2);
\draw (11.5,0) [color=black!50!white]  circle (3pt)[fill] node[below=2pt] {$w_m$};
\draw [color=black!50!white] (11,0) -- (11.5,0);
\draw (7,0) [color=black!50!white] circle (3pt)[fill] node[below =2pt] {$\;\epsilon$};
\draw [color=black!50!white] (7,0) -- (9,0);
\draw (9,0) [color=black!50!white] circle (3pt)[fill] node[below =2pt] {$w_{i+2}$};
\draw[color=black!50!white,decorate,decoration={coil,aspect=0}] (11,0) -- (9,0);
\draw (0.5,-1.5) [color=black!50!white] circle (3pt)[fill] node[below =2pt] {$w_{i+2}+n$};
\draw[decorate,decoration={coil,aspect=0}] (2,2) -- (0,2);
\draw[decorate,decoration={coil,aspect=0}] (2,0) -- (0,0);
\draw (-2,2) circle (3pt)[fill] node[above=2pt] {$0$};
\draw (-2,0) circle (3pt)[fill] node[below=2pt] {$\epsilon$};
\draw (-2,0) -- (-2,2);
\draw (0,2) circle (3pt)[fill] node[above=2pt] {$w_{i-1}$};
\draw (0,0) circle (3pt)[fill] node[below=2pt] {$w_{i+2}$};
\draw (-2,0) -- (0,0);
\draw (-2,2) -- (0,2);
\draw[color=black!50!white,decorate,decoration={coil,aspect=0}] (2.5,-1.5) -- (0.5,-1.5);
\draw (-1.5,-1.5) [color=black!50!white] circle (3pt)[fill] node[below=2pt] {$2\epsilon+3$};
\draw (3,-1.5) [color=black!50!white] circle (3pt)[fill] node[below=2pt] {$\;\;w_{m}+n$};
\draw [color=black!50!white]  (2.5,-1.5) -- (3,-1.5);
\draw [color=black!50!white] (-1.5,-1.5) -- (0.5,-1.5);
\node (P) at (-2.5,1) {$P$};
\draw[->] (-1.7, -0.3) -- (-1.2,-1.2) node[midway, right]{\tiny $+n$};
\draw[->] (3,1) -- (4,1);
\draw[decorate,decoration={coil,aspect=0}] (11,2) -- (9,2);
\draw[red] (7.5,-1.5)  -- (5,2) -- (5,0) -- (7,2);
\draw (5,2) circle (3pt)[fill] node[above=2pt] {$-1$};
\draw (5,0) circle (3pt)[fill] node[below=2pt] {$\epsilon+1$};
\draw (7,2) circle (3pt)[fill] node[above=2pt] {$0$};
\draw (9,2) circle (3pt)[fill] node[above=2pt] {$w_{i-1}$};
\draw[decorate,decoration={coil,aspect=0}] (11.5,-1.5) -- (9.5,-1.5);
\draw (7.5,-1.5) circle (3pt)[fill] node[below=2pt] {$2\epsilon+3$};
\draw (9.5,-1.5) circle (3pt)[fill] node[below=2pt] {$w_{i+2}+n$};
\draw (12,-1.5)  circle (3pt)[fill] node[below=2pt] {$\;\;w_{m}+n$};
\draw (11.5,-1.5) -- (12,-1.5);
\draw (9,2) -- (7,2);
\draw (7.5,-1.5) -- (9.5,-1.5);
\node (Q) at (4.5,1) {$Q$};
\end{tikzpicture}
\end{center}
Considering that $V(\Gamma)=\{0, 1, \ldots, n-3\}$, the paths $Q$ and $Q+n$ are clearly vertex disjoint, hence 
$C_{\infty} = \mathcal{E}\ \cup\  Q \ \cup\ (Q+n)$
is a $(2m+7)$-cycle. Also,
\begin{equation}\label{Cinf1}
     C_{\infty} + n = C_{\infty},
\end{equation}
\begin{equation}\label{Cinf2}
  \begin{aligned}
  V(C_{\infty}) &=  V(Q) \ \cup\ V(Q+n) \ \cup\ \{\infty\} \\
                &=  V(P) \ \cup\ V(P+n) \ \cup\ \{-1, n-2, n-1, -2, \infty\}, \\       
  \end{aligned}
\end{equation}
\begin{equation}\label{Cinf3}
  \begin{aligned}
  \Delta(C_{\infty}) &\supseteq \Delta\mathcal{E}\ \cup\   \Delta Q  = \{\pm n\}\ \cup\   \Delta Q\\
                     &= \pm\{n, n-1, n-2, 2\} \ \cup\ (\Delta P\setminus\{\pm(n-3)\}) .
  \end{aligned}                      
\end{equation}
\begin{center}
\begin{tikzpicture}[x=1cm,y=1cm,scale=0.75]
\draw[decorate,decoration={coil,aspect=0}] (0,2) -- (3,2);
\draw[decorate,decoration={coil,aspect=0}] (0,0) -- (3,0);
\draw[decorate,decoration={coil,aspect=0}] (7,2) -- (10,2);
\draw[decorate,decoration={coil,aspect=0}] (7,0) -- (10,0);
\draw (3,2) -- (4,3) -- (5.5,3) -- (5,3);
\draw [dashed]  (5.5,3) -- (7,2);
\draw (0,2) -- (-1,1);
\draw (0,0) -- (-1,1);
\draw (3,0) -- (4,-1) --  (5.5,-1);
\draw [dashed]  (5.5,-1) -- (7,0);
\draw[red] (5.5, 3) -- (7,0);
\draw[red,] (5.5, -1) -- (7,2);
\draw[->] (10.5,1.5) -- (10.5,.5) node[midway, right]{$n = \epsilon+3$};
\draw (10,0) -- (10,2);
\node(C) at (5,-2.5) {$C_\infty$};
\draw (3,2) circle (3pt)[fill] node[above=2pt] {$0$};
\draw (3,0) circle (3pt)[fill] node[left=2ex, below=2pt] {$\epsilon+3$};
\draw (4,3) circle (3pt)[fill] node[above=2pt] {$\epsilon+1$};
\draw (4,-1) circle (3pt)[fill] node[left=1ex, below=4pt] {\small$2\epsilon+4$};
\draw (5.5,3) circle (3pt)[fill] node[above=2pt] {$-1$};
\draw (5.5,-1) circle (3pt)[fill] node[right=1ex, below=2pt] {$\epsilon+2$};
\draw (7,2) circle (3pt)[fill] node[above=2pt] {$\epsilon$};
\draw (7,0) circle (3pt)[fill] node[right=2ex, below=2pt] {$2\epsilon+3$};
\draw (10,2) circle (3pt)[fill] node[above=2pt] {$w_m$};
\draw (10,0) circle (3pt)[fill] node[below=2pt] {$w_m+n$};
\draw (0,2) circle (3pt)[fill] node[above=2pt] {$w_0$};
\draw (0,0) circle (3pt)[fill] node[below=2pt] {$w_0+n$};
\draw (-1,1) circle (3pt)[fill] node[left=2pt] {$\infty$};
\end{tikzpicture}
\end{center}
We are going to show that the union $\Sigma$ of $C_{\infty}$, $C$, $\Gamma_0$, and $\Gamma_0+n$ is the desired 
$\left[\underline{2m+7}, 4, \;^2 L\right]$-starter of $\Z_{2n}$. By taking into account \eqref{Cinf2}, it is not difficult to check
that the vertex sets of $C_{\infty}$, $C$, $\Gamma_0$, and $\Gamma_0+n$ between them partition 
$\Z_{2n} \ \cup\ \{\infty\}$. Hence $\Sigma\simeq [2m+7, 4, \;^2 L]$ and 
$V(\Sigma) = \Z_{2n} \ \cup\ \{\infty\}$. By \eqref{C4} and \eqref{Cinf1}, it follows that $\Sigma + n = \Sigma$.
Finally, by taking into account \eqref{C4} and \eqref{Cinf3}, and considering that
$\Delta \Gamma_0 = (\{\pm1\} \cup \pm\{4,\ldots, n-3\})\setminus\Delta P$ (see Definition \ref{def:GL}),
it follows that $\Delta\Sigma \supseteq\pm\{1,2,\ldots,n\}\supseteq\Z_{2n}\setminus\{0\}$.
Therefore $\Sigma$ is a 
 $\left[\underline{2m+7}, 4, \;^2 L\right]$-starter  of $\Z_{2n}$.
\end{proof}

The following result follows from \cite[Theorem 4.3]{Buratti Traetta 12}.
For clarity and completeness we include the proof.

\begin{thm}\label{doubling3} Let $K=\{k_1,\ldots, k_t\}$ and 
$K'$ be lists of integers greater than 2, 
  let $\Gamma$ be a graceful labeling of $[K, K', \;^a2  \mid m]_{\epsilon}$, with $a\in\{0,1\}$,
  and denote by $P$ the $m$-path of $\Gamma$, and by $C_i$ a $k_i$-cycle of $\Gamma$, 
  for every $i=1,\ldots, t$. 
  
  Letting $n=\epsilon+2a+1$, if the following two conditions are satisfied,
  \begin{enumerate}  
    \item $(\Delta C_{i} + n)\ \cap\ \Delta P\neq \emptyset$ whenever $1\leq i\leq t$, and
    \item $n-3 = \epsilon \in \Delta P$ when $a=1$,
  \end{enumerate}
  then there is a $\left[\underline{2m+ 4a + 3}, \;^a 4,\; 2K, \;^2K'\right]$-starter.
\end{thm}
\begin{proof} 
By Theorem \ref{doubling1} when $a=0$ and by Theorem \ref{doubling2} when $a=1$, 
there exists a $\left[\underline{2m+ 4a + 3}, \;^a 4,\; \;^2K, \;^2K'\right]$-starter $\Sigma$. 
  Recalling the relative proofs which illustrate how the doubling construction works in each case, and letting
  $C_{\infty}$ be the cycle of $\Sigma$ passing through $\infty$, we have that
  \begin{enumerate}
  \renewcommand{\theenumi}{\emph{\roman{enumi}}}
    \item\label{doubling3:i} 
    $\Delta C_{\infty}\supseteq
           \begin{cases}
             \Delta P & \text{if $a=0$},\\
             \Delta P \setminus\{\pm (n-3)\} & \text{if $a=1$},\\             
           \end{cases}
          $   
    \item\label{doubling3:2} both $C_i$ and $C_{i} + n$ are cycles of $\Sigma$.
  \end{enumerate}
  By assumption $(\Delta C_{i} + n) \ \cap\ \Delta P \neq \emptyset$ whenever $1\leq i\leq t$.
  Note that $\pm (n-3) = (n\pm 3) \not\in (\Delta C_{i} + n)$ when $a=1$, otherwise  
  $\pm 3\in \Delta C_{i}$ contradicting Remark \ref{rem:GL1}. 
  Therefore, by condition (\ref{doubling3:i}) we have that
  \[
    (\Delta C_{i} + n) \ \cap\ \Delta C_{\infty} \neq \emptyset
  \]
  whenever $1\leq i\leq t$. This means that
  there exist a difference $d_i\in \Delta C_i$, 
  an edge $\{a_i, b_i\}$ of $C_i$, 
  and an edge $\{u_i, w_i\}$ of $C_{\infty}$ such that 
  \begin{equation}\label{doubling2:cond1}
    \text{$d_i = b_i-a_i$ and $d_i + n = w_i - u_i$,}
  \end{equation}  
  for every $i=1,2,\ldots, t$. Considering that $C_{\infty} + n = C_{\infty}$ (Remark \ref{starter structure}),
  we have that the edge $\{u_i + n, w_i+n\}$ belongs to $C_{\infty}$, as well.  
  We will use condition \eqref{doubling2:cond1} to turn $\Sigma$ into the desired
  $2$-starter of $\Z_{2n}$. 
  
  Let $\Sigma'$ be the graph obtained from $\Sigma$ by replacing the set of edges
  $\mathcal{E}\subset E(\Sigma)$ with $\mathcal{E}'$, where 
  \begin{align*}  
    \mathcal{E} & = 
    \{\{a_i, b_i\}, \{a_i + n, b_i+n\}, \{u_i, w_i\}, \{u_i + n, w_i+n\}\mid  
    1\leq i\leq t\}, \\
   \mathcal{E}' & = 
    \{\{a_i, b_i+n\}, \{a_i + n, b_i\}, \{u_i, w_i+n\}, \{u_i + n, w_i\}\mid  
    1\leq i\leq t\}    
  \end{align*}
\begin{center}
\begin{tikzpicture}[x=1.3cm,y=1cm,scale=.65]
\draw[->] (-.5,1.5) -- (-.5,.5) node[midway, left]{\small$+n$};
\draw [dashed](0,2) .. controls (1,3.5) and (3,3.5) .. (4,2) node [midway,above] {\tiny$d_i$};
\draw[dashed](0,0) .. controls (1,-1.5) and (3,-1.5) .. (4,0) node [midway,below] {\tiny$d_i$};
\draw [red] (0,2) -- (4,0) node[black, near start,left] {\tiny$d_i+n\hspace{1.5mm}$};
\draw [red] (0,0) -- (4,2) node[black, near start,left=2pt] {\tiny$d_i+n\hspace{1mm}$};
\draw [dashed] (8,2) -- (9.5,2) node [midway, above] {\tiny$d_i+n$};
\draw [dashed] (8,0) -- (9.5,0) node [midway, below] {\tiny$d_i+n$};
\draw (6,0) -- (6,2);
\draw [red] (8,2) -- (9.5,0) node [black,near start, left] {\tiny$d_i$} ;
\draw [red] (8,0) -- (9.5,2) node [black,near start, left] {\tiny$d_i$};

\foreach \y in {2} {
\draw[decorate,decoration={coil,aspect=0}] (4,\y) -- (0,\y);
\draw (0,\y) circle (3pt)[fill] node[above left=2pt] {$a_i$};
\draw (4,\y) circle (3pt)[fill] node[above right=2pt] {$b_i$};
\draw[decorate,decoration={coil,aspect=0}] (6,\y) -- (8,\y);
\draw[decorate,decoration={coil,aspect=0}] (9.5,\y) -- (11.5,\y);
\draw (6,\y) circle (3pt)[fill] node[above=2pt] {};
\draw (8,\y) circle (3pt)[fill] node[above=4pt] {$u_i$};
\draw (9.5,\y) circle (3pt)[fill] node[above=4pt] {$w_i$};
\draw (11.5,\y) circle (3pt)[fill] node[above=2pt] {};
}
\foreach \y in {0} {
\draw[decorate,decoration={coil,aspect=0}] (4,\y) -- (0,\y);
\draw (0,\y) circle (3pt)[fill] node[below=6pt] {$a_i + n \hspace{.5cm}$};
\draw (4,\y) circle (3pt)[fill] node[below=6pt] {$\hspace{.5cm} b_i + n$};
\draw[decorate,decoration={coil,aspect=0}] (6,\y) -- (8,\y);
\draw[decorate,decoration={coil,aspect=0}] (9.5,\y) -- (11.5,\y);
\draw (6,\y) circle (3pt)[fill] node[above=2pt] {};
\draw (8,\y) circle (3pt)[fill] node[below =6pt] {$u_i + n\hspace{.5cm}$};
\draw (9.5,\y) circle (3pt)[fill] node[below =6pt] {$\hspace{.5cm} w_i + n$};
\draw (11.5,\y) circle (3pt)[fill] node[above=2pt] {};
}
\draw (11.5,2) -- (12,1) -- (11.5,0);
\draw (12,1) circle (3pt)[fill] node[right=2pt] {$\infty$};
\node (Ci) at (2,2.5) {$C_i$};
\node (Ci+n) at (2,-.5) {$C_i+n$};
\node (Cinf) at (7,1) {$C_{\infty}$};
\draw [dashed] (4.5, 1.2) -- (5.5,1.2) node [black, midway, above=2pt] {$\mathcal{E}$} ;
\draw [red] (4.5, .8) -- (5.5, .8) node [red, midway, below] {$\mathcal{E}'$} ;
\end{tikzpicture}
\end{center}
  We are going to show that $\Sigma'$ is a 
  $\left[\underline{2m+ 4a + 3}, \;^a 4,\; 2K, \;^2K'\right]$-starter.
  
  Clearly, $V(\Sigma') = V(\Sigma) = \Z_{2n} \cup \{\infty\}$. 
  Considering \eqref{doubling2:cond1}, we have that
  \[
   \Delta \mathcal{E}' = \pm\{n+d_i, n-d_i, 2n+d_i, d_i\} 
   = \pm^2\{n+d_i, d_i\} = \Delta \mathcal{E}.
  \]
  Therefore, recalling that $\Sigma$ is a $2$-starter, it follows that 
  $\Delta \Sigma' = \Delta \Sigma \supset \Z_{2n}\setminus\{0\}$. 

  Further, since $\Sigma + n = \Sigma$ (Definition \ref{2-starter}) 
  and $\mathcal{E} + n = \mathcal{E}$, 
  it follows that $(\Sigma\setminus\mathcal{E}) + n = \Sigma\setminus\mathcal{E}$.
  Since $\Sigma' = (\Sigma\setminus\mathcal{E})\ \cup\ \mathcal{E}'$ and
  $\mathcal{E}' + n = \mathcal{E}'$, we have that $\Sigma' + n = \Sigma'$.
  We have therefore proven that $\Sigma'$ is a $2$-starter of $\Z_{2n}$. 
  
  It is left
  to show that $\Sigma'$ is isomorphic to $[2m+ 4a + 3, \;^a 4,\; 2K, \;^2K']$.  
  Replacing $\mathcal{E}$ with $\mathcal{E}'$ only affects the cycles 
  $C_i$ and $C_i+n$ for $1\leq i\leq t$, 
  and the cycle $C_{\infty}$ of length $2m+ 4a + 3$.
  It is not difficult to see that this substitution turns $C_{\infty}$ into a different cycle
  of the same length, while turns each pair $\{C_i, C_i + n\}$ of $k_i$-cycles into
  a single $(2k_i)$-cycles. Therefore, we can conclude that 
  $\Sigma' \simeq [2m+ 4a + 3, \;^a 4,\; 2K, \;^2K']$, 
  and this completes the proof.
\end{proof}

%
%
%

\section{Labeling extension lemmas}
In this section we show how to extend a labeling of $[L \mid m]$ by either extending the length of the path, or increasing the number of cycles.

The following results show, in two slightly different ways, that if the graph $[L\mid m]$ has a graceful labeling for a particular value of $m$, then the same holds for every sufficiently large $m$. 

\newcommand{\DeltaP}{\Delta_p}

From now on, given a labeling $\Gamma$ of $[L\mid m]$, 
we will denote by $\DeltaP\Gamma$ the list of differences produced by the path in $\Gamma$.

\begin{lem}\label{lem:extension1} 
Let $\Gamma$ be 
a $\GL(x,y)$ $(\text{resp. $\AL(x,y)$})$ of $[L\mid m]_\epsilon$, with $x < y $. Then there exists a 
$\GL(x_{\mu}, y_{\mu})$ $(\text{resp. $\AL(x_{\mu}, y_{\mu})$})$ of $[L\mid m + \mu]$, say $\Gamma_{\mu}$, with
\[  x_{\mu}=
  \begin{cases}
    x& \text{if $\mu$ is even},\\
    \epsilon +\mu -x & \text{if $\mu$ is odd},
  \end{cases}   
  \;\; \text{and}\;\;
  y_{\mu} = y +
  \left\lfloor\frac{\mu}{2}\right\rfloor,
\]
for every $\mu \geq 2x+1$, except possibly when $\mu = 4x+1$. 

Furthermore, $\DeltaP\Gamma_\mu \supseteq \II(\epsilon+1, \epsilon + \mu)$.
\end{lem}
\begin{proof} 
Let $\Gamma$ be a $\GL(x,y)$ of $[L\mid m]_\epsilon$ with $x<y$.
Also,
let $\mu$ be an integer such that
\begin{equation}\label{mu}
  \mu \geq 2x+1\;\;\text{and}\;\; \mu\neq 4x+1. 
\end{equation}  
We start by applying Lemma \ref{paths} with 
$(w_1, \gamma_1) = \left(-\left\lfloor \frac{\mu}{2} \right\rfloor, 
                   \left\lfloor \frac{\mu}{2} \right\rfloor-1\right)$,
$(w_2, \gamma_2) = \left(\epsilon +1, \left\lceil\frac{\mu}{2}\right\rceil-1  \right)$, and 
$i=x$.
Note that
$\gamma_1 < \gamma_2$ if and only if $\mu$ is odd, in which case 
$i\neq \frac{\gamma_2}{2}$ since $\mu\neq 4x+1$ by \eqref{mu}.

Therefore, by Lemma \ref{paths}, the graph $[- \mid \mu-1]$
has an $\AL(J_1, J_2, x', y')$, say $P$, where $J_i = \II(w_i, w_i+\gamma_i)$, for $i=1,2$,
\[ 
  x' =
  \begin{cases}
    w_1 + i = -\frac{\mu}{2} + x & \text{if $\mu$ is even,}\\
    w_2 + \gamma_2 - i = \epsilon +\lceil\frac{\mu}{2}\rceil - x  & \text{if $\mu$ is odd,}
  \end{cases}\;\;\text{and}\;\;
  y' = w_2+i = \epsilon +1 + x.
\]
By joining the $m$-path in $\Gamma$ and the $(\mu-1)$-path $P$ with the edge $\{x, y'\}$, we obtain 
the graph $\Gamma'_{\mu}$ which is clearly isomorphic to $[L \mid m+\mu]$; also,
\begin{align*}
  V(\Gamma'_\mu)    &= V(\Gamma)\ \cup\ V(P) = \II(0, \epsilon) \ \cup \ (J_1\ \cup \ J_2) = 
  \II\left(-\left\lfloor \frac{\mu}{2}\right\rfloor, 
           \epsilon + \left\lceil \frac{\mu}{2}\right\rceil
     \right), \\
 \Delta \Gamma'_\mu &= \Delta \Gamma\ \cup\ \Delta\{x,y'\}\ \cup\ \Delta P = 
                       \Delta \Gamma \ \cup\ \{\pm(\epsilon +1)\} \ \cup\ \pm\II(\epsilon + 2, \epsilon + \mu),
\end{align*}
and the ends of the path in $\Gamma'_\mu$ are $x'$ and $y$.
In other words, 
$\Gamma'_\mu$ is a graceful labeling of $[L\mid m+\mu]$ with parameter set 
$\left(
\II\left(-\left\lfloor \frac{\mu}{2}\right\rfloor, 
\epsilon +\left\lceil \frac{\mu}{2}\right\rceil\right), x', y\right)$. 
Therefore, $\Gamma_\mu = \Gamma'_\mu + \lfloor\frac{\mu}{2}\rfloor$ is the desired graceful labeling of 
$[L \mid m+\mu]$; indeed, the path in $\Gamma_\mu$ has the same list of differences as the one in 
$\Gamma'_\mu$ which contains $\II(\epsilon+1, \epsilon + \mu)$, 
and its ends $x'+\lfloor\frac{\mu}{2}\rfloor$ and $y+\lfloor\frac{\mu}{2}\rfloor$ satisfy the assertion.

Finally, one can check that $\Gamma_\mu$ is an $\alpha$-labeling,
when $\Gamma$ is. Indeed,
if $\Gamma$ is an $\AL\big(\II(0, \delta_1), \II(\delta_1 + 1, \epsilon\big), x, y)$ of $[L\mid m]_\epsilon$,
then $\Gamma'_\mu$ is an 
$\AL\left(J_1, J_2, x, y' \right)$ of $[L\mid m]$, with 
$J_1=\II\left(-\lfloor\frac{\mu}{2}\rfloor, \delta_1\right)$
and 
$J_2 = \II\left(\delta_1 + 1, \epsilon + \left\lceil \frac{\mu}{2}\right\rceil \right)$, hence
$\Gamma_\mu = \Gamma'_\mu + \lfloor\frac{\mu}{2}\rfloor$ is an $\alpha$-labeling.
\end{proof}
\begin{ex} As in the proof of Lemma \ref{lem:extension1}, 
we consider the graphs $\Gamma$ and $P$ pictured below, which respectively represent
a $\GL(x,y)$ of $[4,5\mid 1]_{10}$ with $(x,y)=(5,6)$, and an $\AL\big(\II(-5,-1), \II(11,16), 11, 16\big)$ 
of $[\;- \mid \mu-1]$, with $\mu=11$. By joining the vertices $5$ and $16$ we obtain the graceful labeling $\Gamma_{11}$
of $[4,5\mid 1 + 11]$. The desired labeling is then $\Gamma_{11} + 5$.
\begin{center}
\begin{tikzpicture}[x=1cm,y=1cm,scale=0.75]
\draw[fill=orange!10!white] (-.5,-1) -- (5.5,-1) -- (5.5,3) -- (-.5,3) -- cycle;
\foreach \x in {0, ..., 5} 
{
\coordinate (\x) at (\x,2);
}
\foreach[evaluate={10-\val} as \x] \val in {10, ..., 6} {
\coordinate (\val) at (\x+0.5,0);
}
\foreach \x in {0, ..., 4} {
\draw (\x) circle (3pt)[fill] node[above=2pt] {$\small\x$};
}
\foreach[evaluate={10-\val} as \x] \val in {10,..., 7} {
\draw (\val) circle (3pt)[fill] node[below=2pt] {\small\pgfmathprintnumber{\val}};
}
%
\foreach[evaluate={10-2*\x} as \val, evaluate={10-\x} as \y] \x in {0, 1} {
\draw (\x) -- (\y) node [midway, left] {\tiny\pgfmathprintnumber{\val}};
} 
\foreach[evaluate={9-2*\x} as \val, evaluate={9-\x} as \y] \x in {0, 2} {
\draw (\x) -- (\y) node [midway, left] {\tiny\pgfmathprintnumber{\val}};
}
\foreach[evaluate={11-2*\x} as \val, evaluate={11-\x} as \y] \x in {4, 5} {
\draw (\x) -- (\y) node [near end, right] {\tiny\pgfmathprintnumber{\val}};
}
\draw (2) -- (8) node [near end, left] {\tiny 6};
\draw (10) -- (3) node [midway, above] {\tiny 7};
\draw (4) -- (8) node [near start, left] {\tiny 4};
\draw (1) .. controls (2,3) .. (3) node [near start, above] {\tiny 2};
\draw (2,3.5) [right] node {$\Gamma$};
\draw[fill=orange!10!white] 
(-7,-1) -- (-1,-1) -- (-1,3) -- (-7,3) -- cycle;
\foreach \x in {-7, ..., -1} {
\coordinate (\x) at (\x-1,2);
}
\foreach[evaluate={9.5-\val} as \x] \val in {18, ..., 11} {
\coordinate (\val) at (\x,0);
}
\foreach \x in {-5, ..., -1} {
\draw (\x) circle (3pt)[fill] node[above=2pt] {$\small \x$};
}
\foreach[evaluate={\val-10} as \x] \val in {16,..., 11} {
\draw (\val) circle (3pt)[fill] node[below=2pt] {\small\pgfmathprintnumber{\val}};
}
\foreach[evaluate={11-\x} as \y, evaluate={2*\x-11} as \val] \x in {12,...,16} {
\draw (\x) -- (\y) node[near start, right=1pt, inner sep=0.5pt] {\tiny\pgfmathprintnumber{\val}};
}
\foreach[evaluate={10-\x} as \y, evaluate={2*\x-10} as \val] \x in {11,...,15}{
\draw (\x) -- (\y) node[near end, right=1pt, inner sep=0.5pt] {\tiny\pgfmathprintnumber{\val}};
}
\draw (16) .. controls (-4, -3) and (7.2,-3) .. (5) node[midway, above]{\tiny{11}};
\draw (5) circle (3pt)[fill=white] node[above=2pt] {\small 5};
\draw (6) circle (3pt)[fill=white] node[below=2pt] {\small 6};
\draw (16) circle (3pt)[fill=white];
\draw (11) circle (3pt)[fill=white];
\draw (-4,3.5) [right] node {$P$};
\end{tikzpicture}
\end{center}
\end{ex}

\begin{lem}\label{lem:extension2} Let $\Gamma$ be 
a $\GL(x,y)$ $(\text{resp. $\AL(x,y)$})$ of $[L\mid m]_\epsilon$, with $x < y $. Then there exists a 
$\GL(x_{\mu}, y_{\mu})$ $(\text{resp. $\AL(x_{\mu}, y_{\mu})$})$ of $[L\mid m + \mu]$, say $\Gamma_{\mu}$, with
\[x_{\mu} = x +
  \left\lceil\frac{\mu}{2}\right\rceil, 
  \;\; \text{and}\;\;
  y_{\mu}=
  \begin{cases}
    y + \mu & \text{if $\mu$ is even},\\
    \epsilon -y & \text{if $\mu$ is odd},
  \end{cases}   
\]
for every $\mu \geq 2(\epsilon - y)+1$, except possibly when $\mu = 4(\epsilon - y)+1$. 

Furthermore, $\DeltaP\Gamma_\mu \supseteq \II(\epsilon+1, \epsilon + \mu)$.
\end{lem}
\begin{proof} 
Let $\Gamma$ be a $\GL(x,y)$ (resp. $\AL(x,y)$) of $[L\mid m]_\epsilon$ with $x<y$.
It is not difficult to check that the graph $\Gamma' = -\Gamma + \epsilon$ is 
a $\GL(x',y')$ (resp. $\AL(x',y')$) of~$[L\mid m]$, with
\[
  x' = \epsilon -y<  \epsilon -x = y'.
\]
By applying  Lemma \ref{lem:extension1} to $\Gamma'$, we obtain the existence of  a 
$\GL(x'_{\mu}, y'_{\mu})$ (resp. $\AL(x'_{\mu}, y'_{\mu})$ of $[L\mid m + \mu]$, 
say $\Gamma'_{\mu}$, such that
\[  x'_{\mu}=
  \begin{cases}
    x'& \text{if $\mu$ is even},\\
    \epsilon +\mu -x' & \text{if $\mu$ is odd},
  \end{cases}   
  \;\;
  y'_{\mu} = y' +
  \left\lfloor\frac{\mu}{2}\right\rfloor,\;\;\text{and}
\]
\[
  \DeltaP \Gamma'_\mu \supseteq \II(\epsilon+1, \epsilon + \mu),
\]
whenever $\mu \geq 2x'+1 = 2(\epsilon -y) + 1$, 
except possibly when $\mu = 4x'+1 = 4(\epsilon -y)+1$.
As before, we have that $\Gamma_\mu = -\Gamma'_\mu + (\epsilon + \mu)$ is a 
$\GL(x_{\mu}, y_{\mu})$ (resp. $\AL(x_{\mu}, y_{\mu})$) of $[L\mid m + \mu]$,
with $x_{\mu} = (\epsilon + \mu) - y'_{\mu}$ and 
$y_{\mu} = (\epsilon + \mu) - x'_{\mu}$, which both satisfy the assertion.
Since $\DeltaP \Gamma_\mu = \DeltaP \Gamma'_\mu$, it follows that 
$\Gamma_\mu$ is the desired labeling of $[L\mid m + \mu]$.
\end{proof}

\begin{lem}\label{lem:extension2.5} 
Let $\Gamma$ be a $\AL(x,y)$ of $[L\mid m]_\epsilon$, with $x < y $ and $\epsilon$ odd. 
Then there exists an 
$\AL(x_{\mu}, y_{\mu})$ of $[L\mid m + \mu]$, say $\Gamma_{\mu}$, with
\[x_{\mu} = x +
    \frac{\mu}{2}, 
    \;\; \text{and}\;\;
  y_{\mu}= y + \mu 
\]
for every even $\mu > \epsilon - 2x$, 
\end{lem}
\begin{proof}
Let $\Gamma$ be an $\AL(x, y)$ of $[L\mid m]_{\epsilon}$, where 
$\epsilon$ is an odd integer and $x<y$. Since by assumption $\epsilon$ 
-- which represents the size of $\Gamma$ --
is odd, the two parts of $\Gamma$ as a bipartite graph are the intervals $\II(0, (\epsilon-1)/2)$
and $\II((\epsilon+1)/2, \epsilon)$. Considering also that the cycles of $\Gamma$ have all
even size, then the path in $\Gamma$ has odd length, therefore its end vertices lie in different parts, namely, $x\leq (\epsilon-1)/2 < y$.

For any odd integer $n$, we define the permutation $f_n$ of the interval 
$\II(0, n)$ of even size as follows:
\[
  f_n(i) =
  \begin{cases}
    \frac{n-1}{2} - i  & \text{if $0\leq i \leq \frac{n-1}{2}$},\\
    \frac{3n+1}{2} - i & \text{if $\frac{n+1}{2} \leq i \leq n$}.    
  \end{cases}
\]
One can check that 
$\overline{\Gamma} = f_{\epsilon}(\Gamma)$ 
is an 
$\AL(\ol{x}, \ol{y})$ of $[L \mid m]_{\epsilon}$ where 
\[
\ol{x} = \frac{\epsilon-1}{2} - x,\;\; 
\text{and}\;\; \ol{y} = \frac{3\epsilon+1}{2} - y
\] 
Now let $\mu$ be an even integer with $\mu> \epsilon - 2x$.
This is equivalent to saying that
$\mu > 2\ol{x}+1$. Note also that $\mu\neq 4\ol{x}+1$, since $\mu$ is even. Therefore, 
Lemma \ref{lem:extension1} guarantees that there is an
$\AL(x', y')$ of $[L \mid m + \mu]_{\epsilon + \mu}$, 
say $\Gamma'$, such that
\[
x' =
%
%
\ol{x},
\;\;\;\;
y' = \ol{y} + \frac{\mu}{2}.
\]
It follows that
$\Gamma'' = f_{\epsilon + \mu}(\Gamma')$ 
is the the desired 
$\AL(x'', y'')$ of $[L \mid m + \mu]$, where 
\[
x'' = \frac{\epsilon+\mu-1}{2} - x' = x + \frac{\mu}{2},\;\; \text{and}\;\; 
y'' = \frac{3(\epsilon+\mu)+1}{2} - y' = y + \mu.
\]  
\end{proof}

The following result allows us to construct graceful labelings of
$[L_0,\; L_1\mid m]$ for sufficiently large $m$, 
whenever there exists an $\alpha$-labeling of $[L_0 \mid m_0]$ and a graceful labeling 
of $[L_1 \mid m_1]$.

\begin{lem} \label{lem:extension3} 
Assume there are an $\AL(x_0, y_0)$ of $[L_0 \mid m_0]_{\epsilon_0}$
and a $\GL(x_1, y_1)$ of $[L_1 \mid m_1]_{\epsilon_1}$, where $\epsilon_0$ is odd, and $x_i < y_i$, for $i=0,1$.
Then there exists
a $\GL$ of $[L_0, L_1 \mid m_0 + m_1 + \mu]$, say $\Gamma_\mu$, for every 
$\mu\geq B$ where
\[
  B= 2\epsilon_0 +6\epsilon_1 + 4x_1 - 6y_1 +2x_0 +2y_0 + 16 
     \leq  5\epsilon_0 +6\epsilon_1 +9.
\]
In particular, the result holds for every 
$\mu\geq  5\epsilon_0 +6\epsilon_1 +9$.

Furthermore, 
\begin{align*}
     \DeltaP\Gamma_\mu  
      &\supseteq
       \II(2\epsilon_0 +5\epsilon_1 + 2y_0 + 2x_1 - 4y_1 + 11, \epsilon_0 + \epsilon_1 + \mu)\\
      &\supseteq 
       \II(4\epsilon_0 +5\epsilon_1 + 7, \epsilon_0 + \epsilon_1 + \mu).      
\end{align*}             
\end{lem}
\begin{proof}
Let $\Gamma_0$ be an $\AL(x_0, y_0)$ of $[L_0\mid m_0]_{\epsilon_0}$, where 
$\epsilon_0$ is an odd integer, and
let $\Gamma_1$ be a $\GL(x_1, y_1)$ of $[L_1\mid m_1]_{\epsilon_1}$, 
with $x_i < y_i$ for $i=0,1$.

Set $a=\frac{\epsilon_0+1}{2}-x_0>0$, $b = \epsilon_1 - y_1 + 1>0$, 
$c = y_0 - \frac{\epsilon_0+1}{2} - x_1$.
Also set $\lambda_1(\delta) = \lambda_0(\delta) + c$ where
\begin{equation}\label{lambda0} 
\lambda_0(\delta) = \delta +
\begin{cases}
   -c & \text{if $c\leq -a$},\\
   b-c & \text{if $-a< c \leq b-a$},   \\
   a& \text{if $b-a\leq c  < b$},   \\
   0 & \text{if $c\geq b$}.   
\end{cases}
\;\;\;\text{and}\;\;\;
\delta\in\{0,1\}.
\end{equation}
The labelings we are going to build and
the related parameters will be depending on
$\lambda_0(\delta)$ and $\lambda_1(\delta)$.
However, for clarity reasons, in the following we will write 
$\lambda_0$ and $\lambda_1$ in place of 
$\lambda_0(\delta)$ and $\lambda_1(\delta)$, respectively. 
The dependence on $\delta$ will be understood.

By applying Lemma \ref{lem:extension2.5} to $\Gamma_0$ (with $\mu = 2\lambda_0$), 
and Lemma \ref{lem:extension2} to $\Gamma_1$ (with $\mu = 2\lambda_1$), 
we have that there exist an
$\AL(x'_0, y'_0)$ of $[L_0 \mid m_0 + 2\lambda_0]$
say $\Gamma'_0$, and a
$\GL(x'_1, y'_1)$ of $[L_1 \mid m_1 + 2\lambda_1]$
say $\Gamma'_1$, such that
\begin{equation}\label{endvertices}
x'_i = x_i +  \lambda_i,\;\; \text{and}\;\; 
y'_i = y_i + 2\lambda_i,\;\; \text{for $i=0,1$}.
\end{equation}
We now construct a graceful labeling of 
$[L_0, L_1 \mid m_0+m_1 + 2(\lambda_0+\lambda_1) + 1]_\epsilon$, a zillion graph of size 
$\epsilon = \epsilon_0+\epsilon_1 + 2(\lambda_0+\lambda_1)+1$, 
by modifying and then joining 
the labelings $\Gamma'_0$ and $\Gamma'_1$. More precisely,
let $\Gamma''_0   = \Gamma_0' + (0, \epsilon_1 + 2\lambda_1 +1)$ and 
$\Gamma''_{1} = \Gamma_1' + \left(\frac{\epsilon_0  +1}{2} +\lambda_0\right)$. Note that
\begin{align*}
  V(\Gamma''_0)     &= 
  \II\left(0, \frac{\epsilon_0 -1}{2}+ \lambda_0\right) \ \bigcup\  \\
  &\;\;\;\;\;
  \II\left(\frac{\epsilon_0 + 1}{2} + \lambda_0 + (\epsilon_1 + 2\lambda_1 +1), 
              \epsilon_0 +  2\lambda_0 + (\epsilon_1 +2\lambda_1 +1)\right),\\
  V(\Gamma''_1)     &= 
  \II
  \left(\frac{\epsilon_0 +1}{2} + \lambda_0, 
        \frac{\epsilon_0  + 1}{2} + \lambda_0 + \epsilon_1 + 2\lambda_1
  \right),\\               
  \Delta \Gamma''_0 &= \pm \II\big(\epsilon_1 + 2\lambda_1 +2, \epsilon_0 + \epsilon_1 + 2\lambda_0 +2\lambda_1 +1\big),\\
  \Delta \Gamma''_1 &=
  \begin{cases} 
    \pm \II(1, \epsilon_1 + 2\lambda_1) & \text{if $[L]$ has no $2$-cycle},\\
    \pm \{1, \;^23\}\ \cup\ \pm\II(4, \epsilon_1 + 2\lambda_1) & \text{if $[L]$ has exactly one $2$-cycle}.  
  \end{cases}
\end{align*}
Also, by \eqref{endvertices}, 
the ends $x''_i$ and $y''_i$ of the path in $\Gamma''_i$, for $i=0,1$, are:
\begin{align*}
    x''_0 &= x'_0 = x_0+ \lambda_0,\;\;\;\\
    y''_0 &= y'_0 + (\epsilon_1 +2\lambda_1 +1) = y_0 +\epsilon_1 + 2\lambda_0  +2\lambda_1 +1,\\
    x''_1 &= x'_1 +  \frac{\epsilon_0  +1}{2} +\lambda_0   = 
             x_1+ \frac{\epsilon_0  +1}{2} +\lambda_0 + \lambda_1, \\
    y''_1 &= y'_1 + \frac{\epsilon_0  +1}{2} +\lambda_0   =
             y_1 + \frac{\epsilon_0 +1}{2} + \lambda_0 + 2\lambda_1.
\end{align*}
Furthermore,
\[ y''_0 - x''_1 = \epsilon_1 + \mu_1 +1 + (c + \lambda_0 - \lambda_1) = 
\epsilon_1 + \mu_1 +1,
\]
since $\lambda_1 = \lambda_0 + c$.
Letting $\Gamma = 
\Gamma''_0\ \cup\ \Gamma''_1\ \cup\ \{x''_1,y''_0\}$
be the graph obtained by joining the paths $\Gamma''_0$ and $\Gamma''_1$ through
the edge $\{x''_1,y''_0\}$, we have that 
$\Gamma$ is a labeling of the graph 
$[L_0, L_1 \mid m_0+m_1 + 2(\lambda_0+\lambda_1) + 1]$ of size
$\epsilon = \epsilon_0 +\epsilon_1 + 2\lambda_0  +2\lambda_1 +1$.
Also, 
\[
\begin{aligned}
  & V(\Gamma) = \II(0, \epsilon),\\ 
  & \Delta \Gamma =   \Delta \Gamma''_1
  \ \cup \ \pm 
  \II(\epsilon_1 + 2\lambda_1+1, \epsilon).
\end{aligned}
\]
Therefore, $\Gamma$ is
a $\GL(x''_0, y''_1)$  of  $[L_0, L_1\mid m_0 + m_1 + 2\lambda_0  +2\lambda_1 +1]$.

By  \eqref{lambda0}, it is not difficult to check that 
\[
\begin{aligned}
  0 \leq \lambda_0 = \lambda_0(\delta) &\leq \delta + \frac{\epsilon_0+1}{2} + \epsilon_1 + x_1 - y_1 + 1\\
  &\leq \frac{\epsilon_0+1}{2} + \epsilon_1 + x_1 - y_1 + 2
\end{aligned}  
\]
Therefore, letting $B$ be the parameter defined in the statement, we have that
\begin{equation}\label{B}
\begin{aligned}
  B &= 2\epsilon_0 +6\epsilon_1 + 4x_1 - 6y_1 +2x_0 +2y_0 + 16 \\
  &\geq 6\lambda_0 +  2x_0 +  2y_0 -2x_1 -\epsilon_0 + 1 \\
  & = 6\lambda_0 +  2x_0 +  2c + 2 \\
  & = 4\lambda_0 + 2c + 2(x_0 + \lambda_0) +2 \\
  & =  2(\lambda_0 + \lambda_1) + 2x''_0 + 2.
\end{aligned}  
\end{equation}
Also, 
\[
\begin{aligned}
  \epsilon &=  \epsilon_0 +\epsilon_1 + 2\lambda_0(\delta)  +2\lambda_1(\delta) +1\\
        &=  \epsilon_0 +\epsilon_1 + 4\lambda_0(\delta)  +2c +1\\
        & \leq 2\epsilon_0 +5\epsilon_1 + 2y_0 + 2x_1 - 4y_1 + 10
\end{aligned}  
\]        
Now let $\mu'\geq B$ and set $\mu'' = \mu' - (2\lambda_0 + 2\lambda_1 + 1)$.
By \eqref{B}, we have that $\mu'' \geq 2x_0'' + 1$.
Note also that $4x_0'' + 1 = 4(x_0 + \lambda_0(\delta)) + 1$. Therefore, 
$\mu''$ is distinct from $4x_0'' + 1$ for at least one of the two values of $\delta$.
Recalling also that $x''_0\leq y''_1$, we can apply Lemma \ref{lem:extension1} 
(with $\mu = \mu''$) to $\Gamma$ and obtain
a graceful labeling 
$\Gamma_{\mu}$ of $[L_0, L_1\mid m_0 + m_1  +\mu]$
such that
\begin{align*}
\DeltaP \Gamma_{\mu} &\supseteq \II(\epsilon +1, \epsilon_0 + \epsilon_1 + \mu) \\
&\supseteq
\II(2\epsilon_0 +5\epsilon_1 + 2y_0 + 2x_1 - 4y_1 + 11, \epsilon_0 + \epsilon_1 + \mu).
\end{align*}
Considering that $x_1 - y_1 \leq -1$, $x_0\leq \frac{\epsilon_0-1}{2}$, $y_0\leq \epsilon_0$,
it follows that $B\leq 5\epsilon_0 +6\epsilon_1 +9$, and
this completes the proof.
\end{proof}

\section{Constructing $\alpha$-labelings of $[L\mid m]$}
\label{sec:alpha}
The trivial necessary condition for an $\alpha$-labeling of $[L\mid m]$ to exist is that $[L\mid m]$ is bipartite, namely, all its cycles have even length.
In this section we show that this condition is also sufficient when $m$ is large enough. 

We start by recalling a result proven in \cite{Traetta 13}.

\begin{thm}\label{evencycle+edge}
  For every $\lambda\geq 1$ and $i\in\{0,1\}$, there exists an $\alpha$-labeling of
  $[4\lambda - 2i\mid 2i+1]$ whose $(2i+1)$-path $P$ is as follows:
  \[
    P=\begin{cases}
        \lambda, 3\lambda+1 & \text{if $i=0$},\\
        \lambda-1, 3\lambda+1, \lambda, 3\lambda & \text{if $i=1$},        
      \end{cases}
  \]
  except when $(\lambda,i)\in\{(1,0), (1, 1), (2,1)\}$.
\end{thm}

In the following we generalize Theorem \ref{evencycle+edge}. 
More precisely, we show (Theorem \ref{evencycle+path}) that the graph $[\ell \mid m]$, 
with $\ell+m \equiv 1\pmod{4}$, has an 
$\alpha$-labeling with specific end vertices for the $m$-path, whenever
$m\geq  \ell + 5$. Lemmas \ref{0mod4cycle+path} and \ref{2mod4cycle+path} 
deal with the cases
$8\leq \ell\equiv 0 \pmod{4}$ and $10\leq \ell\equiv 2 \pmod{4}$, respectively,
by making use of  Theorem \ref{evencycle+edge}. The cases where $m\in\{4,6\}$, corresponding to the exceptions 
$(\lambda,i)\in\{(1,0), (2,1)\}$ in Theorem \ref{evencycle+edge}, are dealt with separately
in Lemmas \ref{4cycle+path} and  \ref{6cycle+path}.

\begin{lem}\label{0mod4cycle+path}
  There exists an $\AL(\lambda+\mu, 3(\lambda+\mu)+1)$  of $[4\lambda \mid 4\mu+1]$ 
  whenever $\lambda\geq 2$ and, either $\mu=0$ or $\mu\geq \lambda+1$.
\end{lem}
\begin{proof} 
   By Theorem \ref{evencycle+edge}, there exists an $\AL(\lambda, 3\lambda+1)$ of
   $[4\lambda\mid 1]$, say $\Gamma$. Note that $\Gamma$ is the desired $\alpha$-labeling when $\mu=0$, hence 
   we can assume that $\mu\geq \lambda+1$. 
   
   By Lemma \ref{paths} there exists an $\AL(J_{i,1}, J_{i,2}, x_i, y_i)$ of $[\;-\mid 2\mu-1]$, say $P_i$,
   whose parameters are the following:     
   \[
     \begin{array}{|c||c|c|}
     \hline \rule{0pt}{1\normalbaselineskip}  
       i &  1   &  2  \\[0.5ex] \hline \hline \rule{0pt}{1\normalbaselineskip}
       J_{i,1}  &  \II(0,\mu-1)  &  \II(2\lambda+\mu+1, 2\lambda+2\mu)       
       \\[0.5ex] \hline \rule{0pt}{1\normalbaselineskip}
       J_{i,2}  & \II(4\lambda+ 3\mu+2, 4\lambda+ 4\mu +1) & \II(2\lambda +2\mu +1, 2\lambda +3\mu)
       \\[0.5ex] \hline \rule{0pt}{1\normalbaselineskip}
       x_i     & \lambda   & 3\lambda+\mu+1  
       \\[0.5ex] \hline \rule{0pt}{1\normalbaselineskip}
       y_i     &  5\lambda+3\mu+2    &  3\lambda +2\mu+1         
       \\ [0.5ex] \hline 
     \end{array}
   \]   
   Also, $\Gamma'= \Gamma + (\mu, 3\mu)$ is an $\AL(J'_1, J'_2, x',y')$ of $[4\lambda \mid 1]$ where
   \begin{align*}
     J'_1=\II(\mu, 2\lambda +\mu),\;\;\; J'_2=\II(2\lambda+ 3\mu +1, 4\lambda +3\mu +1),\\
    \text{and}\; (x',y') = (\lambda + \mu, 3\lambda +3\mu +1).
   \end{align*} 
   Considering that the vertex sets of $P_1$, $P_2$ and $\Gamma'$ partition 
   $\II(0, 4\lambda +4\mu +1)$, and letting 
   $\mathcal{E}=\big\{\{x', y_1\}, \{y', x_2\}\big\}$, 
   the graph $H = \mathcal{E}\ \cup\ P_1\ \cup\ P_2\ \cup\ \Gamma'$ (obtained 
   by extending the $1$-path in $\Gamma'$ using $P_1$, $P_2$ and the edges in 
   $\mathcal{E}$) is  an 
   $\AL(x_1, y_2)$ of $[4\lambda \mid 4\mu +1]$. Indeed, $H$ is a bipartite graph whose parts are the intervals
   $J_{1,1} \cup J_{2,1} \cup J'_1 = \II(0, 2\lambda +2\mu)$ and 
   $J_{1,2} \cup J_{2,2} \cup J'_2 = \II(2\lambda +2\mu +1, 4\lambda+4\mu+1)$.  Also, 
   \begin{align*}
     \Delta H &= \Delta \mathcal{E}\ \cup\ \Delta P_1\ \cup\ \Delta P_2\ \cup\ \Delta \Gamma' = \\
              &=\pm \big( \{2\mu, 4\lambda +2\mu+2\} \ \cup\ 
                          \II(4\lambda +2\mu+3, 4\lambda +4\mu+1) \ \cup\  \II(1, 2\mu-1) \\
              &\;\;\;\;\;\;\;\;\; \cup\  \II(2\mu+1, 4\lambda +2\mu +1)\big) =\pm \II(1, 4\lambda +4\mu +1).  
   \end{align*}
   Since $y_2 - x_1 = y' - x'$, we obtain the desired
   $\alpha$-labeling of $[4\lambda \mid 4\mu+1]$ from $H$ by 
   replacing the edge $\{x', y'\}$ 
   -- originally belonging to $\Gamma'$ and hence to the path in $H$ -- 
   with the edge $\{x_1, y_2\}$ joining the ends of the path in $H$.
\end{proof}

\begin{lem}\label{2mod4cycle+path}
   There exists an $\AL(\lambda+\mu, 3(\lambda+\mu)+1)$  of $[4\lambda -2 \mid 4\mu+3]$ 
  whenever $\lambda\geq 3$ and, either $\mu=0$ or $\mu\geq \lambda$.
\end{lem}
\begin{proof} We proceed similarly to the proof of Lemmna \ref{0mod4cycle+path}.
 By Theorem \ref{evencycle+edge}, there exists an $\alpha$-labeling of
   $[4\lambda-2\mid 3]$, say $\Gamma$, whose $3$-path is $\lambda-1, 3\lambda+1, \lambda, 3\lambda$. 
   When $\mu=0$, the desired $\alpha$-labeling 
   of $[4\lambda-2\mid 3]$ is obtained  from $\Gamma$ by replacing
   the edge $\{3\lambda+1, \lambda\}$ with $\{\lambda-1, 3\lambda\}$. 
   
   We now assume that
   $\mu\geq \lambda$. 
   By Lemma \ref{paths} there exists an $\AL(J_{i,1}, J_{i,2}, x_i, y_i)$ of $[\;-\mid 2\mu-1]$, say $P_i$,
   whose parameters are the following:      
   \[
     \begin{array}{|c||c|c|}
     \hline \rule{0pt}{1\normalbaselineskip}  
       i &  1   &  2  \\[0.5ex] \hline \hline \rule{0pt}{1\normalbaselineskip}
       J_{i,1}  &  \II(0,\mu-1)  &  \II(2\lambda+\mu+1, 2\lambda+2\mu)       
       \\[0.5ex] \hline \rule{0pt}{1\normalbaselineskip}
       J_{i,2}  & \II(4\lambda+3\mu+2, 4\lambda+4\mu+1) & \II(2\lambda +2\mu +1, 2\lambda +3\mu)
       \\[0.5ex] \hline \rule{0pt}{1\normalbaselineskip}
       x_i     & \lambda -1   & 3\lambda+\mu
       \\[0.5ex] \hline \rule{0pt}{1\normalbaselineskip}
       y_i     &  5\lambda+3\mu+1    &  3\lambda +2\mu         
       \\ [0.5ex] \hline 
     \end{array}
   \]      
   Also, $\Gamma'= \Gamma + (\mu, 3\mu)$ is an $\AL(J'_1, J'_2, x',y')$ of $[4\lambda-2 \mid 3]$ where
   \begin{align*}
     J'_1=\II(\mu, 2\lambda +\mu),\;\;\; J'_2=\II(2\lambda+ 3\mu +1, 4\lambda +3\mu +1),\\
    \text{and}\; (x',y') = (\lambda +\mu -1, 3\lambda+3\mu).
   \end{align*} 
   Furthermore, the $3$-path of $\Gamma'$ is $Q'= x', y'', x'', y'$, where
   $y'' = 3(\lambda+\mu)+1$ and
   $x'' = \lambda +\mu$.
   Considering that the vertex sets of $P_1$, $P_2$ and $\Gamma'$ partition 
   $\II(0, 4\lambda +4\mu +1)$, and letting 
   $\mathcal{E}=\big\{\{x', y_1\}, \{y', x_2\}\big\}$, 
   the graph $H = \mathcal{E}\ \cup\ P_1\ \cup\ P_2\ \cup\ \Gamma'$  is an 
   $\AL(x_1, y_2)$ of $[4\lambda -2 \mid 4\mu +3]$. 

   Since $y_2 - x_1 = y'' - x''$, we obtain the desired
   $\alpha$-labeling of $[4\lambda -2 \mid 4\mu+3]$ from $H$ by 
   replacing the edge $\{x'', y''\}$ 
   -- belonging to the path in $H$ -- 
   with the edge $\{x_1, y_2\}$.
\end{proof}

\begin{lem}\label{4cycle+path}
   There exists an $\AL(\mu+1, 3\mu+4)$  of $[4 \mid 4\mu+1]$ 
   for every $\mu\geq 2$.
\end{lem}
\begin{proof} When $\mu\in\{2,3\}$ the desired $\alpha$-labeling is the union of the cycle $C$ and the path $Q$ given below:
\begin{align*}
   \mu=2: &\; C= (4, 7, 6,  8),\; Q = 10, 5, 13, 0, 12, 1, 11, 2, 9, 3,\\
   \mu=3: &\; C= (6, 9, 8, 10),\; Q = 4, 12, 7, 14, 5, 11, 0, 17, 1, 16, 2, 15, 3, 13.
\end{align*}
   Now let $\mu\geq 4$. We set $C= (\mu + 2, 3\mu+1, \mu+4, 3\mu+2)$ and construct
   a $(4\mu+1)$-path $Q$, disjoint from $C$, by joining the paths
   $P_1,  P_2$ and $P_3$ defined below: 
   \begin{align*}
     P_1 & = 3\mu + 4, \mu + 3, 3\mu - 1, \\
     P_3 & = 2\mu+1, 4\mu+5, 0, \ldots, 4\mu+5 - j, j, \ldots, 3\mu+5, \mu, 3\mu+3, \mu + 1,
   \end{align*}  
   while $P_2$ is an $\AL(J_{1}, J_{2}, 2\mu+1, 3\mu-1)$ of 
   $[\;-\mid 2\mu-5]$, with 
   $J_1 = \II(\mu+5, 2\mu+2)$ and $J_2 = \II(2\mu+3, 3\mu)$, 
   which exists by Lemma \ref{paths}.
     
   We claim that $C \ \cup\ Q$ is the desired 
   $\AL(\mu+1, 3\mu+4)$  of $[4 \mid 4\mu+1]$. 
   Indeed, $C \ \cup\ Q$ is a bipartite graph whose parts are 
   $\II(0, 2\mu+2)$ and $\II(2\mu+3, 4\mu+5)$. 
   Also, the ends of $Q$ are $\mu+1$ and $3\mu+4$.
   Finally, considering that
   \[\Delta C= \pm\II(2\mu-3, 2\mu),\;\; \Delta P_1 = \pm\{2\mu-4, 2\mu+1\}, \;\;
   \Delta P_2 = \pm \II(1, 2\mu-5)
   \]
   \[\Delta P_3 = \pm \II(2\mu+4, 4\mu+5)\ \cup\ \pm \{2\mu+2, 2\mu+3\},
   \]
   it follows that $\Delta(C \ \cup\ Q) = \II(1, 4\mu+5)$,
   and this completes the proof.
\end{proof}
\begin{lem}\label{6cycle+path}
   There exists an $\AL(\mu+2, 3\mu+7)$  of $[6 \mid 4\mu+3]$ 
   for every $\mu\geq 2$.
\end{lem}
\begin{proof} When $\mu=4$, the desired $\alpha$-labeling is the union of the cycle $C$ and the path $Q$ given below:
\begin{align*}
   & C= (2, 23, 4, 20, 5, 22),\\
   & Q = 6, 18, 9, 15, 10, 17, 7, 21, 3, 25, 0, 24, 1, 14, 11, 13, 12, 16, 8, 19.
\end{align*}
   Now let $\mu\geq 2$ with $\mu\neq 4$. 
   By Lemma \ref{paths}, there is an $\AL(J_{i,1}, J_{i,2}, x_i, y_i)$ of $[\;-\mid 2\mu]$, say $P_i$,
   whose parameters are the following:      
   \[
     \begin{array}{|c||c|c|}
     \hline \rule{0pt}{1\normalbaselineskip}  
       i &  1   &  2  \\[0.5ex] \hline \hline \rule{0pt}{1\normalbaselineskip}
       J_{i,1}  &  \II(0,\mu)  &  \II(\mu+5, 2\mu+4)       
       \\[0.5ex] \hline \rule{0pt}{1\normalbaselineskip}
       J_{i,2}  & \II(3\mu+10, 4\mu+9) & \II(2\mu +5, 3\mu+5)
       \\[0.5ex] \hline \rule{0pt}{1\normalbaselineskip}
       x_i     &  2   & 2\mu +7
       \\[0.5ex] \hline \rule{0pt}{1\normalbaselineskip}
       y_i     &  \mu-2    &  3\mu+3         
       \\ [0.5ex] \hline 
     \end{array}
   \]      
   Consider the $6$-cycle $C$ and the matching $\mathcal{E}$ defined as follows:
   \begin{align*}
     & C =(\mu+1, 3\mu+9, \mu+3, 3\mu+6, \mu+4, 3\mu+8), \\
     & \mathcal{E}= \{\{x_1, x_2\}, \{y_1, 3\mu+7\}, \{y_2, \mu+2\}\}
   \end{align*}
   and let $H$ be the union of $P_1, P_2, \mathcal{E}$ and $C$. 
     
   We claim that $H$ is the desired 
   $\AL(\mu+2, 3\mu+7)$  of $[6 \mid 4\mu+3]$. Indeed, $H$ is a bipartite graph whose parts are
   $J_1 = \II(0,2\mu+4)$ and $J_1 = \II(2\mu+4, 4\mu+9)$. 
   Also, $P_1 \ \cup\  P_2\ \cup\ \mathcal{E}$ is a $(4\mu+3)$-path, disjoint from $C$, whose ends are  
   $\mu+2$ and $3\mu+7$.  One can finally check that $\Delta H = \pm\II(1, 4\mu+9)$, and this completes the proof.
\end{proof}

The above four lemmas can be summarized as follows.

\begin{thm}\label{evencycle+path}
  Let $\ell$ and $m$ be integers with $\ell\geq 4$ even, $m\geq 1$ odd and $\ell+m \equiv 1\pmod{4}$. 
  Then there exists an 
  $\AL\left(\frac{\ell+m-1}{4}, 3\left(\frac{\ell+m-1}{4}\right) + 1\right)$ 
  of $[\ell \mid m]$
  when either $m\in\{1,3\}$ and $\ell\geq 8$, or $m\geq  \ell + 5$.
\end{thm}

We are now able to prove the main results of this section.

\begin{thm}\label{thm:main1:1}
Let $L$ be a nonempty list of even integers greater than~$2$.
Then there is an 
$\AL\left(x, y\right)$
of $\left[L \;\Big|\; \epsilon -\sum_{\ell\in L} \ell \right]$ where\\
\[
\epsilon = 2|L|(\max(L)+3)-1, \;\;\;
x = \frac{\max(L)}{2}+1,\;\; \;
y=\frac{\epsilon+\max(L)+3}{2}.
\]
\end{thm}
\begin{proof} Let $L=\{\ell_0, \ell_1, \ldots, \ell_t\}$ and set $k=2\max(L)+5\equiv1\pmod{4}$.
By Theorem \ref{evencycle+path}, the graph $[\ell_i\ \mid k-\ell_i]$ has 
an $\AL\left(\frac{k-1}{4}, 3\left(\frac{k-1}{4}\right) + 1\right)$, say~$\Gamma_i$,
for every $i\in\II(0,t)$. 
Recalling that the parts of  $\Gamma_i$ (as a bipartite graph) are $J_{i,1}=\II(0,w-1)$ and 
$J_{i,2}=\II(w, 2w-1)$ where $w=\frac{k+1}{2}$, 
we have that $\Gamma'_i = \Gamma_i + (iw, (2t-i)w)$ is an $\AL(J'_{i,1}, J'_{i, 2}, x'_{i}, y'_i)$
of  $[\ell_i\ \mid k-\ell_i]$ where
\begin{enumerate}
  \item $J'_{i, 1} = J_{i,1} + iw = \II\big(iw, (i+1)w-1\big)$,
  \item $J'_{i, 2} = J_{i,2} + (2t-i)w = \II\big((2t-i+1)w,  (2t-i+2)w -1\big)$,
  \item $x'_i = \frac{k-1}{4} + iw            \in J'_{i, 1}$,
  \item $y'_i = 3\left( \frac{k-1}{4}\right) + 1 + (2t-i)w \in J'_{i, 2}$;
  \item $\Delta \Gamma'_i = \pm \II(2(t-i)w +1, 2(t-i)w +2w -1)$,
\end{enumerate}
for every $i\in\II(0,t)$. Finally, let $\mathcal{E} =\big\{\{y'_{i}, x'_{i+1}\}\mid i\in \II(0,t-1)\big\}$.

\begin{minipage}{15cm}
\hspace{-2cm}
\begin{tikzpicture}[x=.8cm,y=1cm,scale=0.75]

\foreach[evaluate={10-\val} as \x] \val in {10, ..., 6} {
\coordinate (\val) at (\x+0.5,0);
}

\foreach \j in {0, 6, 14, 20} 
{  
  \INTEGERQUOTIENT{\j}{6}{\i};
  \coordinate  (B1\i) at (0+\j,-.2); 
  \coordinate  (B2\i) at (5+\j,-.2);
  \coordinate  (B3\i) at (5+\j,2.2);
  \coordinate  (B4\i) at (0+\j,2.2);        
\draw[fill=orange!10!white] (B1\i) -- (B2\i) -- (B3\i) -- (B4\i) -- (B1\i); 

  \coordinate  (SV1\i) at (0.3+\j,2.05); 
  \coordinate  (SV2\i) at (4.7+\j,2.05);
  \coordinate  (SV3\i) at (4.7+\j,1.75);
  \coordinate  (SV4\i) at (0.3+\j,1.75);      
\draw[fill=orange!10!white] (SV1\i) -- (SV2\i) -- (SV3\i) -- (SV4\i) -- (SV1\i); 
\draw[->] (.5+\j, 1.9) -- (.5+\j, 2.5);

  \coordinate  (LV1\i) at (0.3+\j,.25); 
  \coordinate  (LV2\i) at (4.7+\j,.25);
  \coordinate  (LV3\i) at (4.7+\j,-.05);
  \coordinate  (LV4\i) at (0.3+\j,-.05);      
\draw[fill=orange!10!white] (LV1\i) -- (LV2\i) -- (LV3\i) -- (LV4\i) -- (LV1\i); 
\draw[->] (.5+\j, 0.1) -- (.5+\j, -.5);

  \coordinate  (b\i) at (2.5+\j, 0.1);           
  \coordinate  (a\i) at (2.5+\j, 1.9) {};
  \draw [in=30, out=-150, looseness=3.25] (a\i) to (b\i);
  \draw[black,thick,dotted] (2.5+\j,1) ellipse (1.7cm and .5cm);   
  \draw (a\i) circle (2pt)[fill];		
  \draw (b\i) circle (2pt)[fill];	 
 
} 
  \draw (a0) node[above=5pt] {$\small x'_0$};		
  \draw (a1) node[above=5pt] {$\small x'_{1}$};	
  \draw (a2) node[above=5pt] {$\small x'_{t-1}$};	
  \draw (a3) node[above=5pt] {$\small x'_{t}$};	     
  \draw (b0) node[below=5pt] {$\small y'_0$};	  
  \draw (b1) node[below=5pt] {$\small y'_{1}$};	  
  \draw (b2) node[below=5pt] {$\small y'_{t-1}$};	  
  \draw (b3) node[below=5pt] {$\small y'_t$};	        

\draw[red, in=170, out=-10] (b0) to (a1);  
\draw[red, in=230, out=-10, dashed] (b1) to (12,1.5);
\draw[red, in=50, out=170, dashed] (a2) to (13,.5);  
\draw[red, in=170, out=-10] (b2) to (a3);

  \node[red, rotate=45] (e1) at (5.5,1.3) {\tiny $2tw$};
  \node[red, rotate=45] (e2) at (11.5,1.3) {\tiny $\hspace{3mm}(2t-2)w$};
  \node[red, rotate=40] (e3) at (13.5,1.3) {\tiny $4w$};  
  \node[red, rotate=45] (e3) at (19.5,1.3) {\tiny $2w$};   
  
  \node (SInt0) at (.5, 2.8) {\small $J'_{0,1}$};  
  \node (SInt1) at (.5+6, 2.8) {\small $J'_{1,1}$};  
  \node (SInt2) at (.5+14, 2.8) {\small $J'_{t-1,1}$};  
  \node (SInt3) at (.5+20, 2.8) {\small $J'_{t,1}$};    
  
  \node (LInt0) at (.5, -.8) {\small $J'_{0,2}$};  
  \node (LInt1) at (.5+6, -.8) {\small $J'_{1,2}$};  
  \node (LInt2) at (.5+14, -.8) {\small $J'_{t-1,2}$};  
  \node (LInt3) at (.5+20, -.8) {\small $J'_{t,2}$};        
  
  \node (G0) at (3.5,1) {\small $\Gamma'_{0}$};  
  \node (G1) at (3.5+6,1) {\small $\Gamma'_{1}$};  
  \node (G2) at (3.5+14, 1) {\small $\Gamma'_{t-1}$};  
  \node (G3) at (3.5+20, 1) {\small $\Gamma'_{t}$};     
 
\end{tikzpicture}
\end{minipage}

We claim that $H = \bigcup_{i=0}^t \Gamma'_i \ \cup \ \mathcal{E} $ is the desired $\alpha$-labeling.
Indeed, $H$ is a bipartite graph whose parts are 
$J'_1 = \bigcup_{i=0}^t J'_{i, 1}$ and $J'_2 = \bigcup_{i=0}^t J'_{i, 2}$. 
Also, it is isomorphic to the graph
$\left[L \;\Big|\; \epsilon - \sum_{i=0}^{t} \ell_i\right]$ of size
$\epsilon = 2(t+1)w -1$,
since the edges in $\mathcal{E} $ join
the paths of the $\Gamma'_i$s to form a single path of the desired length. 
Furthermore, the ends of the path in $H$ are 
$x'_0= \frac{k-1}{4} = \frac{\max(L)}{2}+1$ and 
$y'_t=3\left( \frac{k-1}{4}\right) +tw + 1 = \frac{\epsilon+\max(L)+3}{2}$.
Finally, since 
\[y'_{i}- x'_{i+1} = \textstyle{\frac{k-1}{2}} + 1 + (2t-2i-1)w = 2(t-i)w,\] 
it follows that
$\Delta H = \bigcup_{i=0}^t \Delta \Gamma'_i \ \cup \ \Delta \mathcal{E}  = \pm\II(1, \epsilon)$,
and this completes the proof.
\end{proof}

Theorem \ref{thm:main1:1} and Lemma \ref{lem:extension1} result in the following.
\begin{thm}\label{thm:main1:2}
  Let $L$ be a nonempty list of even integers greater than~$2$.    
  If $\epsilon \geq 2(|L|+1)(\max(L)+3)-1$, 
  then there is an $\AL$ of 
  $\left[L \;\Big|\; \epsilon - \sum_{\ell\in L} \ell\right]$, say $\Gamma$.
  Furthermore,
  $\DeltaP \Gamma \supseteq \II\big(2|L|(\max(L)+3),\; \epsilon\big)$.
\end{thm}
\begin{proof}
Let $\epsilon = \epsilon_0 + \mu$, where $\epsilon_0 = 2|L|(\max(L)+3)-1$
and $\mu\geq 2(\max(L) + 3)$.

By Theorem  \ref{thm:main1:1}, there exists 
an $\AL\left(x_0, y_0\right)$
of $\left[L \;\Big|\; \epsilon_0 -\sum_{\ell\in L} \ell \right]$, 
where $x_0 = \frac{\max(L)}{2}+1$. 
Considering that $\mu\geq 4x_0+2$, 
Lemma \ref{lem:extension1} guarantees that 
$\left[L \;\Big|\; \epsilon_0 + \mu -\sum_{\ell\in L} \ell \right]$ 
has an $\AL$, say $\Gamma$, 
such that 
$\DeltaP \Gamma \supseteq \II(\epsilon_0+1, \epsilon_0+\mu)$. 
The assertion follows.
\end{proof}

The following corollary provides a lower bound on $m$, 
as a function of  $|L|$ and $\max(L)$, 
which guarantees the existence of an $\alpha$-labeling of $[L\mid m]$.
\begin{cor}\label{cor:main1}
  Let $L$ be a list even integers greater than~$2$. 
  Then there exists an $\alpha$-labeling of 
$\left[L \mid m\right]$ for every  $m\geq 2(|L|+1)(\max(L) + 1) +3$.
\end{cor}
\begin{proof}
  The result easily  follows from Theorem \ref{thm:main1:2} 
  once we notice that $\sum_{\ell\in L} \ell \geq 4|L|$.
\end{proof}


\section{Constructing graceful labelings of $[L\mid m]$}
\label{sec:beta}
In this section we provide lower bounds on $m$ which guarantee that
$[L\mid m]$ has a graceful labelings. We start by considering the cases
where $[L]$ contains only even or odd cycles, and deal with the general case at the end.

\subsection{The case with only even cycles}
As mentioned in Section \ref{subsec:alpha}, a classic $\alpha$-labeling is a 
graceful labeling. Therefore for every sufficiently large $m$,
the existence of a graceful
labeling of the graph $[L\mid m]$ free from $2$-cycles
is shown in Theorem \ref{thm:main1:2}.

Here we deal with the case where $[L\mid m]$ has exactly one $2$-cycle.
\begin{thm}\label{thm:main2even:1}
Let $L$ be a list of even integers greater than~$2$. 
Then there is an 
$\AL\left(x, y\right)$
of $\left[L, 2 \;\Big|\; \epsilon -\sum_{\ell\in L} - 2 \right]$, 
where\\
\[
\epsilon = (2|L|+1)(\max(L)+3), \;\;\;
x = \frac{\max(L)}{2}+1,\;\; \;
y = (|L|+1)(\max(L)+3)-1.
\]
\end{thm}
\begin{proof} If $L$ is empty, the desired $\GL$ of $[2\mid 1]$ is $\Gamma = C\ \cup\ P$ with
the $2$-cycle $C= (0,3)$ and the $1$-path $P= 2,3$. 

We now assume that $L$ is nonempty 
and let $\epsilon_0 = 2|L|(\max(L) + 3) - 1$, $a = \max(L)+4$ and $b = \frac{\epsilon_0+1}{2}$. 
  By Theorem \ref{thm:main1:1}, there exists
  an $\AL\left(x_0, y_0\right)$ of $\left[L \;\Big|\; \epsilon_0 -\sum_{\ell\in L} \ell \right]$,
  say $\Gamma_0$,  where $x_0  = a/2-1$, and $y_0  = x_0 +b$.
  Recalling that the partite sets of $\Gamma_0$ are 
  $J_{01}=\II\left(0, b-1\right)$ and 
  $J_{02}=\II\left(b, \epsilon_0\right)$,
  the graph $\Gamma'_0 = \Gamma_0 + (0, a)$
  is an $\AL$ of 
  $\left[L \;\Big|\; \epsilon_0 -\sum_{\ell\in L} \ell \right]$  
  with parameter set $\left(J'_{01}, J'_{02}, x'_0, y'_0\right)$, where
  
  \[ 
    J'_{01} = J_{01} = \II\left(0, b-1\right),\;\; 
    J'_{02} = J_{02}+a = \II\left(a+b,  a + \epsilon_0\right),
  \]
  \[
    x'_0 = x_0 = a/2-1, \;\;\text{and}\;\; y'_0 = y_0+a =  3a/2+b-1.
  \]
  
  Now consider the graceful labeling $(0,3)\ \cup\ \{1,2\}$
  of $[2 \mid 1]$. By applying to this labeling 
  Lemma \ref{lem:extension2} (with $\mu = a-4$),
  we obtain the existence of a $\GL\left(x_1, y_1\right)$ of $[2\mid a-3]$, 
  say $\Gamma_1$, where $x_1 = a/2-1$ and $y_1=a-2$.
  Recalling that $V(\Gamma_1)=\II(0, a-1)$,
  by remark \ref{rem:GL2} the graph $\Gamma'_1 = \Gamma_1 + b$
  is a $\GL\left(J_1', x'_1, y'_1\right)$ of $[2\mid a-3]$, where
  \[
    J_1' = \II(b, a+b-1),\;\; x'_1 = x_1 + b = a/2 + b -1, \;\; y'_1 = y_1 + b = a + b -2.
  \]
  
  We are going to show that $\Gamma = \Gamma'_0\ \cup\ \Gamma'_1\ \cup\ \{x'_1, y'_0\}$
  is the desired graceful labeling. Since $\Gamma'_0$ and $\Gamma'_1$ are vertex disjoint, 
  and the edge $\{x'_1, y'_0\}$ joins the respective paths to form one 
  of length $m = a + \epsilon_0 -\sum_{\ell\in L} \ell - 2$, it follows that
  $\Gamma$ is isomorphic to $[L, 2\mid m]$. 
  Also, $V(\Gamma) = J'_{01}\ \cup\ J_1'\ \cup\ J'_{02} = \II\left(0, a + \epsilon_0\right)$.
  Furthermore, 
  \begin{align*}
  \Delta\Gamma &= \Delta\Gamma'_0 \ \cup\ \Delta\Gamma'_1 \cup\ \pm\{y'_0 - x'_1\} \\
  &= \pm\II(a+1, a+\epsilon_0) \ \cup\  \pm\II(1, a-1) \ \cup\ \{\pm a\} 
   = \II(1,  a+\epsilon_0).
  \end{align*}
  Therefore, $\Gamma$ is a $\GL(x, y)$ of 
  $\left[L, 2 \mid \epsilon -\sum_{\ell\in L} \ell - 2\right]$, where
  \[\epsilon = a + \epsilon_0 = (2|L|+1)(\max(L) + 3),\]
  \[
  x=x'_0 = \frac{\max(L)}{2}+1 ,\;\; y=y'_1=(|L|+1)(\max(L)+3)-1,
  \]
  and the assertion is proven.
\end{proof}

Similarly to the proof of Theorem \ref{thm:main1:2}, one can easily check that
Theorem \ref{thm:main2even:1} and Lemma \ref{lem:extension1} result in the following.
\begin{thm}\label{thm:main2even:2}
  Let $L$ be a list of even integers greater than~$2$.    
  If $\epsilon \geq (2|L|+3)(\max(L)+3)$, 
  then there is an $\GL$ of 
  $\left[L, 2 \;\Big|\; \epsilon - \sum_{\ell\in L} \ell -2\right]$, say $\Gamma$.
  Furthermore,
  $\DeltaP \Gamma \supseteq \II\big((2|L|+1)(\max(L)+3) + 1,\; \epsilon\big)$.
\end{thm}

\subsection{The case with only odd cycles and at most one $2$-cycle}

The following result shows that we can always extend a graceful labeling of 
$[L\mid m]_\epsilon$ by adding a cycle no longer than $\epsilon+1$, 
and suitably increasing the length of the path.
This will be the base ingredient of a recursive construction that will allow us to prove 
the existence of a graceful 
labeling of $[L\mid m]$, containing only odd cycles, as long as $m$ is sufficiently large.

\begin{lem}\label{lem:odd cycles}
 Assume there is a $\GL(x, y)$ of  $[L\mid m]_{\epsilon}$, with  $x<y$, let
 $t$ be a positive integer such that $0 \leq x \leq \epsilon-2t$ 
 and set $m'=6\epsilon  + 2(x - y) + 6$. Then 
 there exists a $\GL(x', y')$ of 
 $[L, 2t + 1 \mid  m + m']$, say $\Gamma'$, where
 \begin{enumerate}
   \item\label{lem:odd cycles:2} $x' \leq \frac{m'}{2}$,
   \item\label{lem:odd cycles:3} $y' - x' = 2\epsilon + t + 2$,
   \item\label{lem:odd cycles:4} $\Delta_p\Gamma'\supseteq 
  \II(\epsilon + 2t+1, \epsilon + 2t+1 + m')\setminus\{y' - x'\}$.
 \end{enumerate}   
\end{lem}
\begin{proof}
 Let $\Gamma$ be a $\GL(x, y)$ of  $[L\mid m]_{\epsilon}$, where $0\leq x<y$ and 
 $V(\Gamma)=\II(0, \epsilon)$, and let $P$ denote the $m$-path in $\Gamma$.
 Also, recall that  
 \[  
  \Delta \Gamma = 
  \begin{cases} 
    \pm \II(1, \epsilon) & \text{if $[L \mid m]$ has no $2$-cycle},\\
    \pm \{1, \;^23\}\ \cup\ \pm\II(4, \epsilon) & \text{if $[L \mid m]$ has exactly one $2$-cycle}.  
  \end{cases}
 \]
 
 Now consider the cycle 
 $C=(-1, \epsilon+1, \ldots, -i, \epsilon+i, \ldots, -t, \epsilon+t, -\epsilon-2)$ 
 of length $2t+1$, 
 where $t>0$ and $\epsilon-2t \geq x$. Note that 
 \begin{align*}
      V(C) &= \{-\epsilon - 2\}\ \cup \ \II(-t, -1) \ \cup \ \II(\epsilon+1, \epsilon+t),\\
  \Delta C &= \pm \II(\epsilon+1, \epsilon +2t)   \ \cup \ \{\pm(2\epsilon+t+2)\}.
 \end{align*}
 
 Lemma \ref{paths} guarantees the existence of an $\AL(I_h, J_h, x_h, y_h)$ of $[\;-\mid m_h]$, say $P_h$ for $h\in \{1,2\}$, whose parameters are the following
 \[
   \begin{array}{|c||c|c|c|}
     \hline \rule{0pt}{1\normalbaselineskip}  
       h  &       1                                           &       2                       
         \\[0.5ex] \hline \hline \rule{0pt}{1\normalbaselineskip}
      x_h &  y -3\epsilon -4                                & x   +t -\epsilon -1      
         \\[0.5ex] \hline \rule{0pt}{1\normalbaselineskip}
      y_h & x +t + 2\epsilon +2                             & x  +2t +\epsilon +1           
         \\[0.5ex] \hline \rule{0pt}{1\normalbaselineskip}         
      I_h  &   \II(-(x -y +t +3\epsilon +4), -\epsilon -3)   & \II(-\epsilon -1, -t-1)
         \\[0.5ex]  \hline \rule{0pt}{1\normalbaselineskip}
      J_h &    \II(2\epsilon +2, x - y +t +4\epsilon +3) & \II(\epsilon + t +1, 2\epsilon +1)  
         \\[0.5ex]  \hline \rule{0pt}{1\normalbaselineskip}
      m_h &   2(x - y +t)+4\epsilon+3 &  2\epsilon - 2t + 1      
      \\ [0.5ex] \hline 
   \end{array}
 \]  
Note that $P_2$ is well defined, namely, $x_2\in I_2$ and $y_2\in J_2$, 
since by assumption $0\leq x \leq \epsilon -2t$. 
Also,
\[ \Delta P_h = 
\begin{cases}
  \pm \II(3\epsilon + 5, 2(x - y + t) + 7\epsilon + 7) & \text{if $h = 1$}, \\    
  \pm \II(2t +\epsilon +2, 3\epsilon+2)                & \text{if $h = 2$}.
\end{cases} 
\]
Now let 
$P_{2,1} = x_2, \ldots, u $ and
$P_{2,2} = v, \ldots, y_2 $ be
the paths that we obtain from $P_2$ by removing the edge $\{u,v\}$ that gives the differences 
$\pm(2\epsilon+t+2)$ already covered by $C$. 
Clearly,  $V(P_{2,1}\ \cup\ P_{2,2}) = V(P_2)$ and $\Delta (P_{2,1}\ \cup\ P_{2,2}) = \Delta P_2\setminus\{\pm(2\epsilon+t+2)\}$.

The graphs $\Gamma, C, P_1, P_{2,1},$ and $P_{2,2}$ are vertex disjoint, and denoting by 
$\Gamma'$ their union, we have that
\begin{align*}
  V(\Gamma')     &= \II(-x +y -t -3\epsilon -4, x - y +t +4\epsilon +3)\\
  \Delta \Gamma' &= \Delta\Gamma\ \cup\ \pm \II(\epsilon+1, 2(x - y + t) + 7\epsilon + 7)\setminus\Omega,
\end{align*}
where $\Omega = \pm \{\epsilon +2t +1,\; 3\epsilon+3,\; 3\epsilon+4\}$.

We obtain the desired graceful labeling from $\Gamma'$, by adding the edges in 
$\mathcal{E}=\big\{ \{y_2, x\},\; \{y, x_1\},\; \{y_1, x_2\}\big\}$, and 
then considering a suitable translate. 
In fact,
the matching $\mathcal{E}$
joins $P, P_1, P_{2,1}, P_{2,2}$ to form a single path of length $m + m'$, 
with $m' = 6\epsilon + 2(x - y) +6$,
whose end vertices are $u$ and $v$. Since
$\Delta \mathcal{E} = \Omega$, we have that $\Gamma' \cup\,  \mathcal{E}$ 
is a $\GL(J', u, v)$ of $[L, 2t+1 \mid m +m']$, with $J' = V(\Gamma')$.

Finally, let $a=x -y +t +3\epsilon +4$. Recalling that $P_2$ is a bipartite graph whose parts are $I_2$ and $J_2$, we can assume (without loss of generality) that $u\in I_2$ and $v\in J_2$.
One can check that $(\Gamma' \cup\,  \mathcal{E})+a$ is a $\GL(x', y')$ of $[L, 2t+1 \mid m +m']$, 
with $(x', y') = (u+a, v+a)$, which satisfies conditions
\ref{lem:odd cycles:2}--\ref{lem:odd cycles:4}.
\end{proof}

\begin{ex} In the picture below we consider the $\GL(x,y)$ $\Gamma$ of $[3 \mid 4]$ and following  
the proof of Lemma \ref{lem:odd cycles} we construct 
a $\GL(u,v)$ of $[3, 7 \mid 48]$ by adding the 7-cycle $C$ and connecting the path in $\Gamma$
with the paths $P_1, P_{2,1}$ and $P_{2,2}$ through the red edges which form the set $\mathcal{E}$ in the proof.
Its  translate by 26 provide the final graceful labeling satisfying the assertion.
\begin{center}
\includegraphics[width=\textwidth]{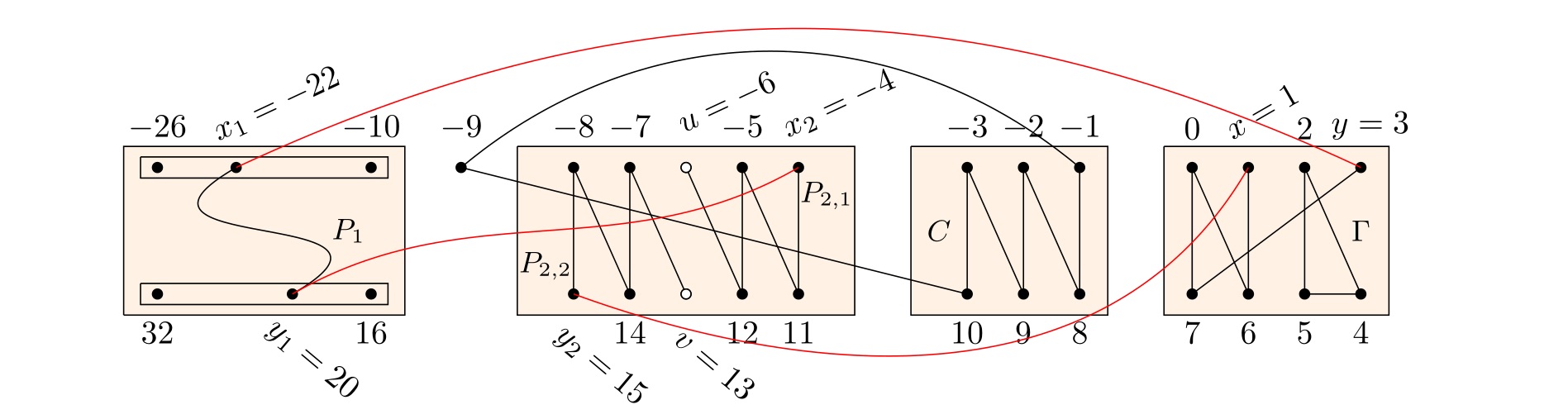}
\end{center}
\end{ex}

\begin{lem}\label{lem:base case:1}
  For every odd $\ell\geq 3$, there exists $\epsilon\in\II(\ell, 2\ell +1)$ such that 
  $[\ell\mid \epsilon-\ell]$ has a $\GL(x,y)$ which satisfies the following conditions:
  \begin{enumerate}
    \item $0\leq x\leq \min\left(\frac{\epsilon-1}{2}, \epsilon -\ell\right)$, and
    \item if $\ell\geq 7$, then $y - x \geq (\ell-1)/2+4$.   
  \end{enumerate} 
\end{lem}
\begin{proof} Let $\ell\geq 3$ be an odd integer. If $\ell\in\{3,5,7,9\}$, 
the desired graceful labeling $\Gamma = C\ \cup\ P$ of $[\ell \mid \epsilon-\ell]$ is given below.
\begin{align*} 
  (\ell,\epsilon) = (3,6): &\; C=(2,3,5),\;\;      P=1,6,0,4;\\
  (\ell,\epsilon) = (5,6): &\; C=(0,6,2,3,5),\;\;      P=1,4;\\
  (\ell,\epsilon) = (7,14): &\; C=(3,11,4,10,7,6,8),\;\;P=2,12,1,13,0,14,5,9;\\
  (\ell,\epsilon)= (9,17): &\; C=(0,16,1,15,2,9,8,6,12),\;\; P=3,14,4,13,5,10,7,11.
\end{align*}
Now assume that $\ell\geq 11$. In \cite{Traetta 13} it is shown that $[\ell\mid 1]$ has a 
$\GL(x', y')$ with $x'=\left\lfloor \frac{\ell+1}{4}\right\rfloor$ and
 $y'=\left\lfloor \frac{3\ell+3}{4}\right\rfloor$.
By applying Lemma \ref{lem:extension1} (with $\mu=2x'+2$), we obtain the existence of 
a $\GL(x, y)$ of $[\ell \mid 2x'+3]_\epsilon$, with $(x,y)=(x', x'+y'+1)$, which satisfies the assertion.
\end{proof}

\begin{lem}\label{lem:base case:2}
  For every odd $\ell\geq 3$, there exists $\epsilon\in\II(\ell+2, 7(2\ell +1))$ such that 
  $[2,\ell\mid \epsilon-\ell-2]$ has a $\GL(x,y)$ which satisfies the following conditions:
  \begin{enumerate}
    \item $0\leq x\leq \min\left(\frac{\epsilon-1}{2}, \epsilon -\ell\right)$, and
    \item $y - x \geq (\ell-1)/2+4$.   
  \end{enumerate} 
\end{lem}
\begin{proof} 
Let $\ell\geq 3$ be an odd integer. If $\ell\in\{3,5,7,9, 11, 13, 19\}$, 
the desired graceful labeling $\Gamma = C_0 \ \cup\ C_1 \ \cup\ P$ 
of $[2, \ell \mid \epsilon-\ell-2]$ is given below.
\begin{align*} 
  (\ell,\epsilon) = (3,12): &\; C_0=(5, 8),\;\; C_1=(2,6,7),\;\;      P=9,0,12,1,11,3,10,4;\\
  (\ell,\epsilon) = (5,14): &\; C_0=(6,9),\;\; C_1=(2,10,4,8,7),\;\;  
                                P=11,0,14,1,13,3,12,5;\\
  (\ell,\epsilon) = (7,16): &\; C_0=(7,10),\;\; C_1=(3,12,4,11,5,9,8),\\  
                            &\; P=13,2,14,1,15,0,16,6;\\
  (\ell,\epsilon) = (9,26): &\; C_0=(12,15),\;\; C_1=(7,19,8,18,9,17,10,16,11),\\  
                            &\; P=14,13,26,0,25,1,24,2,23,3,22,4,21,5,20,6;\\
  (\ell,\epsilon) = (11,22): &\; C_0=(10,13),\;\; C_1=(6,18,5,16,7,17,9,15,8,12,11),\\  
                             &\; P=14,0,22,1,21,2,20,3,19,4;\\  
  (\ell,\epsilon) = (13,32): &\; C_0=(15,18),\;\; C_1=(9,24,8,22,10,23,12,21,11,19,14,20,13),\\  
                             &\; P=17,16,32,0,31,1,30,2,29,3,28,4,27,5,26,6,25,7;\\                                                        
  (\ell,\epsilon) = (19,53): &\; C_0=(26,29),\\ 
              &\;   C_1=(13,40,14,39,15,38,16,37,17,36,18,35,19,34,20,33,21,\\
              &\;\;\;\;\;\;\;\;\;\;\;32,22),\\                                                           
              &\;     P=28,27,23,31,24,30,25,53,0,52,1,51,2,50,3,49,4,48,\\  
              &\;\;\;\;\;\;\;\;\;5,47,6,46,7,45,8,44,9,43,10,2,11,41,12.\\                                                        
\end{align*}

Now let $\ell$ be an odd integer with $\ell\geq 15$ and $\ell\neq 19$.
It will be useful to write 
$\ell = 2t+1$ where $t = 2\tau + \rho +1$ with 
$\tau\geq 3$, $\rho\in\{0,1\}$ and $(\tau,\rho)\neq (4,0)$.

Let $C_0, C_1$ and $P_0$ be the $2$-cycle, the $(2t+1)$-cycle and the $(2t+5)$-path, respectively, defined below:
\[ C_0 = (\tau-1, \tau+2),\]
\[C_1 = (-t-1, 2t-1, \ldots, -t-1+i,2t-1-i, \ldots, -2, t, -1),\]
\[P_0 = 2t, -t-2,\ldots,  2t+i, -t-i-2, \ldots, 3t+1, -2t-3, 3t+2, 1.\]
Note that
\begin{equation}\label{eq:vertices1}
  \begin{aligned}
        V(C_1) &= \II(-t-1, -1)\ \cup\ \II(t, 2t-1),\;\;\text{and}\\
        V(P_0) &= \II(-2t-3, -t-2)\ \cup\ \{1\} \ \cup\  \II(2t, 3t+2).      
  \end{aligned}
\end{equation}  
Also, 
\begin{equation}\label{eq:diff1}
    \Delta C_0 = \;^2\{\pm3\}, \;\;\Delta C_1 = \II(t, 3t), \;\;\text{and}\;\;
    \Delta P_0 = \II(3t+1, 5t+5).
\end{equation}

We now apply Lemma \ref{paths} with 
$(w_1, \gamma_1) = (0, \tau-2)$,
$(w_2, \gamma_2) = (\tau+3, \tau-3+\rho)$, and $i=1$. 
Note that $\gamma_1>\gamma_2$
if and only if $\rho=0$, in which case
$i\neq \frac{\gamma_2}{2}$ since $(\tau,\rho)\neq (4,0)$.
Therefore, Lemma \ref{paths} guarantees that the graph $[- \mid t-5]$
has an $\AL(J_{11}, J_{12}, x_1, y_1)$, say $P_1$, where 
\begin{equation}\label{eq:vertices2}
  \begin{aligned}
   J_{11} = \II(w_1, w_1+\gamma_1) = \II(0, \tau-2),\\
   J_{12} = \II(w_2, w_2+\gamma_2) = \II(\tau+3, t-1), \\
  x_1= w_1 + i = 1, \;\;\;
  y_1=
  \begin{cases}
      w_1+\gamma_1-i = \tau-3 & \text{if $\rho=0$},\\
      w_2 + i = \tau+4 & \text{if $\rho=1$}.   
  \end{cases}
  \end{aligned}
\end{equation}
We recall that by definition of $\alpha$-labeling we have that
\begin{equation}\label{eq:diff2}
  \Delta P_1 = \II(5, t-1).
\end{equation}
Finally, let $P_2=y_1, \tau+1-\rho, \tau+\rho$.

By \eqref{eq:vertices1} and \eqref{eq:vertices2}, it follows that
$Q=P_0 \ \cup\ P_1 \ \cup\ P_2$ is a $(3t+2)$-path 
whose end vertices are $\tau+\rho$ and $2t$, 
and the graph $\Gamma = C_0 \ \cup\ C_1 \ \cup\ Q$ is isomorphic to 
$[2, 2t+1 \mid 3t+2]$, with $V(\Gamma) = \II(-2t-3, 3t+2)$. 
Also, 
by \eqref{eq:diff1} and \eqref{eq:diff2} and considering that $\Delta Q = \pm\{1,4\}$, 
we have that $\Delta\Gamma = \II(1, 5t+5)$.
Therefore, $\Gamma' = \Gamma + (2t+3)$ is a $\GL(x, y)$ of
$[2, \ell\mid \epsilon-\ell-2]$, where $x=2t+3 + \tau+\rho$,  $y=4t+3$ and $\epsilon = 5t+5$. 
One can easily check that
$x\leq \epsilon - \ell$ and $y - x = 2t - \tau -\rho \geq t+4,$
therefore, $\Gamma'$ satisfies the assertion.
\end{proof}

\begin{thm}\label{thm:main2odd:1}
 Let $L$ be a list of odd integers greater than 1, 
 let $a\in\{0,1\}$, 
 and set $\epsilon_1=7^{|L|+a-1}\big(2\max(L) +1\big)$ when $L\neq \emptyset$, 
 otherwise $\epsilon_1=3$.
 Then there exists 
 $\epsilon\in \II\left(\sum_{\ell\in L} \ell +2a + 1, \epsilon_1\right)$
 such that
 $\left[L, \;^a2 \mid \epsilon - \sum_{\ell\in L}\ell - 2a\right]_\epsilon$ has a 
 $\GL(x,y)$ where  $0\leq x\leq \frac{\epsilon-1}{2}$ and $x<y$.
\end{thm}
\begin{proof} If $L$ is empty, then $\Gamma_0=\{0,1\}$  and $\Gamma_1=(0,3)\ \cup\ \{1,2\}$
are graceful labelings of $[\,-\mid 1]$ and $[2\mid 1]$, respectively, that satisfy the assertion with $\epsilon_1=3$.

Now let $L=\{\ell_0, \ell_1,\cdots, \ell_u\}$ be a nonempty list of odd integers, with
$\ell_0\geq \ell_1 \geq \cdots \geq \ell_u\geq 3$, 
set $L_h = \{\ell_0, \ell_1, \ldots, \ell_h\}$ and let $\ell_h = 2t_h + 1$ 
for $0\leq h\leq u$. 
Since the case $u=0$ is proven in Lemmas \ref{lem:base case:1} and \ref{lem:base case:2}, 
we can assume that $u>0$. 

We prove by induction that for every $h\in\II(0,u)$, 
with $h\geq 1$ when $a=0$ and $\ell_0\leq5$, there exists a $\GL(x_h, y_h)$
of $\left[L_h, \;^a2\;\big|\; \epsilon_h - \sum_{j=0}^h \ell_j - 2a\right]_{\epsilon_h}$ where
\begin{enumerate}
  \item \label{thm:main2odd:cond1} 
  $ 2a + 1+\sum_{j=0}^h \ell_j\leq \epsilon_h\leq  7^{h+a}\big(2\max(L)+1\big)$;
  \item \label{thm:main2odd:cond2}
  $x_h \leq \min\left(\frac{\epsilon_h-1}{2}, \epsilon_h -\ell_h\right)$;
  \item \label{thm:main2odd:cond3}
  $y_h - x_h \geq t_h +4$. 
\end{enumerate}
If $\ell_0\geq 7$ or $a=1$, the base case $h=0$ is proven in 
Lemmas \ref{lem:base case:1} and \ref{lem:base case:2}. 
If $\ell_0\in\{3,5\}$ and $a=0$, the base case is $h = 1$ and the desired graceful labeling of 
$[\ell_0, \ell_1\mid \epsilon_1 -\ell_0- \ell_1]$ is the graph
$\Gamma = C_0\ \cup\ C_1\ \cup\ P$, where the $\ell_i$-cycle $C_i$, for $i=0,1$, and the path $P$ are given below:
  \begin{align*}
    (\ell_0, \ell_1, \epsilon_1)= (3,3,13):\; &C_0=(5,6,8),\;   C_1=(2,7,11),\; P=4,10,3,13,0,12,1,9;\\
    (\ell_0, \ell_1, \epsilon_1)= (5,3,19):\; &C_0=(3,10,11),\; C_1=(7,13,8,12,9),\\
                                            &\;\,P=14,0,19,1,18,2,17,4,16,5,15,6;\\  
    (\ell_0, \ell_1, \epsilon_1)= (5,5,39):\;  &C_0=(14, 25, 15, 24, 27),\; C_1=(16, 23, 17, 22, 18),\\
    &\;\,P= 13, 21, 20, 8, 31, 9, 30, 10, 29, 11, 28, 12, 26, 2, 37, 3, \\
    &\;\;\;\;\;\;\;\;\;\, 36, 0, 39, 1, 38, 6, 33, 7, 32, 4, 35, 5, 34, 19.
  \end{align*}
Now supposing the assertion is true for some $i\in \II(0, u-1)$,  we prove it holds
for $i+1$. By the induction hypothesis, the graph $[L_i, \; ^a2 \mid m_i]$ 
has a $\GL(x_i, y_i)$, where $m_{i} = \epsilon_{i} - \sum_{j=0}^{i} \ell_j - 2a$, 
which satisfies conditions \ref{thm:main2odd:cond1}--\ref{thm:main2odd:cond3} with $h=i$.
Hence $\Gamma_i$ satisfies the assumptions of Lemma \ref{lem:odd cycles}, since 
$0 \leq x_i \leq \epsilon_i-\ell_i < \epsilon_i -2t_{i+1}$. Therefore, 
$[L_{i+1},  \; ^a2\mid  m_i + m_{i+1}]_{\epsilon_{i+1}}$, with $m_{i+1}= 6\epsilon_i + 2(x_i- y_i) +6$,
has a $\GL(x_{i+1}, y_{i+1})$, say $\Gamma_{i+1}$, such that
    \[
      x_{i+1} \leq m_{i+1}/2,
      \;\;\text{and}
      \;\; y_{i+1} - x_{i+1} = 2\epsilon_i + t_{i+1} + 2\geq  t_{i+1} + 4.
    \]
Since $\epsilon_{i+1} = \epsilon_i + \ell_{i+1} + m_{i+1}$, 
it follows that $x_{i+1} \leq \min(\epsilon_{i+1}/2-1,\epsilon_{i+1} -\ell_{i+1})$, 
hence $\Gamma_{i+1}$ satisfies conditions 
\ref{thm:main2odd:cond2} and \ref{thm:main2odd:cond3} with $h=i+1$.  
Finally, since $x_i- y_i\leq -t_i-4$, we have that 
\begin{align*}
   \epsilon_{i+1}
                  &= 7\epsilon_i + \ell_{i+1} + 2(x_i- y_i) +6
                     \leq 7\epsilon_i + \ell_{i+1} -\ell_i -1< 7\epsilon_i.
\end{align*}
Therefore,  $\Gamma_{i+1}$ satisfies condition \ref{thm:main2odd:cond1} with $h=i+1$, and this completes the proof.
\end{proof}

Theorem \ref{thm:main2odd:1} and Lemma \ref{lem:extension1} result in the following.

\begin{thm}\label{thm:main2odd:2}
 Let $L$ be a list of odd integers greater than 1, and let $a\in\{0,1\}$.
 Then there exists a graceful labeling $\Gamma_\epsilon$ of 
 $\left[L, \;^a2 \mid \epsilon - \sum_{\ell\in L}\ell - 2a\right]_\epsilon$
 for every $\epsilon \geq 3\epsilon_1$, where 
 $\epsilon_1 =  7^{|L|+a-1}\big(2\max(L) +1\big)$ if $L\neq \emptyset$, otherwise $\epsilon_1=3$. Furthermore, 
 $\DeltaP \Gamma_\epsilon \supseteq 
   \II
   \left( 
      \epsilon_1+1, \epsilon
   \right)$.
\end{thm}
\begin{proof}
By Theorem \ref{thm:main2odd:1},
there is a $\GL(x,y)$ of $\left[L, \;^a2 \mid \epsilon' - \sum_{\ell\in L}\ell - 2a\right]_{\epsilon'}$
for some $\epsilon'\leq \epsilon_1$, 
such that $0\leq x\leq \frac{\epsilon'-1}{2}$.

Let $\epsilon \geq 3\epsilon_1$, and set $\mu = \epsilon - \epsilon'$.
Since  
\[\mu \geq 3\epsilon_1 - \epsilon_1 = 2\epsilon_1 \geq 2 \epsilon' \geq 4x+2,\] 
by Lemma \ref{lem:extension1} there is 
a $\GL$ of $\left[L, \;^a2 \mid \epsilon + \sum_{\ell\in L}\ell - 2a\right]$, 
say $\Gamma$, such that $\DeltaP \Gamma \supseteq \II(\epsilon_1+1, \epsilon)$.
\end{proof}


\subsection{The general case}

Here, we apply Lemma \ref{lem:extension3} to Theorems \ref{thm:main1:1} and \ref{thm:main2odd:1}, 
and construct graceful labelings of $[L \mid m]$ and $[L, 2 \mid m]$ whenever
$L$ contains both odd cycles and even cycles of length at least $4$, as long as 
$m$ is large enough.

\begin{thm}\label{thm:main2general:1}
 For $i=0,1$, let $L_i$ be a list of integers congruent to $i \pmod{2}$ and greater than $2$, 
 and let $a\in \{0,1\}$.
 Then there exists a graceful labeling of 
 $\left[L_0,  L_1, \;^a2 \mid \epsilon - \sum_{\ell\in L_0\cup L_1} \ell -2a\right]$ 
 for every $\epsilon \geq 6\epsilon_0 + 7\epsilon_1 +9 $, where 
 \begin{align*}
   & \epsilon_0 = \max\Big(1,\; 2|L_0|\big( \max(L_0) + 3 \big) - 1\Big),\\ 
   & \epsilon_1 = \max\left(3,\; 7^{|L_1|+a-1}\big(2 \max(L_1)+1\big)\right).   
 \end{align*} 
 Furthermore, $\DeltaP\Gamma \supseteq
 \II(4\epsilon_0 + 5\epsilon_1 +7, \epsilon)$.
 \end{thm}
\begin{proof} The existence of an $\alpha$-labeling $\Gamma_0$ of 
$[L_0\mid m_0]_{\epsilon_0}$ is trivial
when $L_0$ is empty, in which case $m_0 = \epsilon_0=1$ and then $\Gamma=\{0,1\}$; otherwise, 
$\Gamma_0$ exists by Theorem \ref{thm:main1:1}.
  Also, by Theorem \ref{thm:main2odd:1}, there exists a graceful labeling $\Gamma_1$ of 
  $[L_1,\,^a2 \mid m'_1]_{\epsilon'_1}$ for some positive
  $\epsilon'_1 \leq \epsilon_1$. 
  Note that $\epsilon_0$ is odd, and the paths in both $\Gamma_0$ and $\Gamma_1$ have distinct end vertices. 
  Therefore, by Lemma \ref{lem:extension3} there exists a graceful labeling $\Gamma$
  of $\left[L_0, L_1, \;^a2 \mid m_0 + m'_1 + \mu\right]_\epsilon$ 
  such that 
  $\DeltaP\Gamma\supseteq \II(4\epsilon_0 + 5\epsilon'_1 +7, \epsilon_0 + \epsilon'_1 +\mu)$,
  for every $\mu\geq 5\epsilon_0 + 6\epsilon'_1 +9$. 
  Considering that the size of $\Gamma$ is $\epsilon = \epsilon_0 + \epsilon'_1 +\mu$
  and that $\epsilon'_1\leq \epsilon_1$,
  the assertion easily follows.
\end{proof}

Note that lower bounds on the path length better than the one in Theorem \ref{thm:main2general:1} are given in Theorems \ref{thm:main1:2} and \ref{thm:main2even:2} when $L_1=\emptyset$,
and in Theorem \ref{thm:main2odd:2} when $L_0=\emptyset$.

Note that Theorem \ref{mainGL} is a consequence of Theorem \ref{thm:main2general:1} 
when $a=0$.


\section{New results on the Oberwolfach problem}  

In this section we construct solutions to the Oberwolfach problem $OP^\s(\lambda, F)$ 
whenever $F$ has a sufficiently large cycle, and 
either $F$ is a single-flip $2$-regular graph or $\lambda$ is even. 
These results are obtained by applying Theorem \ref{doubling3} to the $\alpha$-labelings
and graceful labelings built in Sections \ref{sec:alpha} and \ref{sec:beta}, respectively.

We start by proving Theorem \ref{mainOP}.
\\

\noindent
\textbf{Theorem \ref{mainOP}.}
\emph{ 
Let $F=[h, 2\ell_1, \ldots, 2\ell_r, \;^{2}\ell_{r+1}, \ldots, \;^{2}\ell_{s}]$ 
where
$2\leq \ell_1< \ell_2< \ldots < \ell_r$ and 
$h,\ell_{r+1},\ell_{r+2}, \ldots, \ell_{s} \geq 3$,
and let $I=\{i\mid 2<\ell_i\; \text{is even}, 1\leq i\leq s\}$.
Then $OP^\s(F)$ has a $(1-\s)$-pyramidal solution whenever 
\[h> 16\max(1,h_0) + 20\max(3,h_1) + 29,\] where
 \begin{align*}
   & h_0 = 2|I|\max\{\ell_i + 3\mid i\in I\}  - 1,\\
   & h_1 = 
   7^{s-|I|-1}\max\{2\ell_i+1\mid i\not\in I\;\text{and}\; \ell_i\neq 2\}. 
 \end{align*}
}
\begin{proof} It is helpful to represent the single-flip $2$-regular graph $F$ by suitably collecting the parameters $\ell_i$. 
  Let $K=\{\ell_i \mid 3\leq \ell_i, 1\leq i\leq r\}$ and $K'=\{\ell_{r+1}, \ell_{r+2}, \ldots, \ell_{s}\}$, and set $a=\left\lfloor 2/\ell_1\right\rfloor$. By assumption 
  $2\leq\ell_1< \ell_i$ for $2\leq i\leq r$, hence $a\in\{0,1\}$ and 
  $\ell_1\not\in K \Leftrightarrow \ell_1=2 \Leftrightarrow a=1$. 
  Therefore we can write $F= [h, \;^a4, 2K, \;^2K']$.
  
  Set $\epsilon_0=\max(1,h_0)$, $\epsilon_1=\max(3,h_1)$ 
  and assume that $h\geq 12\epsilon_0 + 14\epsilon_1 + 21$. 
  It is enough to show that 
  if $h$ is odd, then there is an $F$-starter (see Definition \ref{2-starter}) whose cycle passing through $\infty$ has length $h$. 
  Indeed,  it follows by Theorem \ref{1rotsols} that there is a $1$-rotational solution to $OP(F)$, 
  while the cases where $h$ is even and $\sigma\in\{-1,1\}$ are a straightforward consequence of 
  Theorem \ref{2pyrsols}.

  First, let $L_0 = \{\ell_i\mid i\in I\}$ and
  $L_1 = \{\ell_i\mid i\not\in I\;\text{and}\; \ell_i\neq 2\}$,
  and note that $L_0\ \cup\ L_1 = K\ \cup\ K'$. Let $m= (h-1)/2 -2a -1$ 
  (recall that  we are supposing  that $h$ is odd) and consider the graph
  $[L_0,  L_1, \;^a2 \mid m]_\epsilon$ of size 
  \begin{equation}\label{mainOP:eq}
    \epsilon = 
     \sum_{\ell\in L_0\cup L_1} \ell +  2a  + m \geq (h-1)/2 -1 
     \geq 8\epsilon_0 + 10\epsilon_1 +13.  
  \end{equation}
  By Theorem \ref{thm:main2general:1} there exists a graceful labeling $\Gamma$ of
  $[L_0,  L_1, \;^a2 \mid m]_\epsilon$ such that
  \[\DeltaP\Gamma \supseteq \II(3, \epsilon)\setminus I, \;\;\text{where}\;\;
  I=\II(3, 4\epsilon_0 + 5\epsilon_1 + 6).\]
  Therefore $\Gamma$ clearly satisfies condition 2 of Theorem \ref{doubling3}.
  Also, for every cycle $C$ of $\Gamma$, 
  we have that $-I\ \cap\ \Delta C\neq \emptyset$, hence
  $(-I+n)\ \cap\ (\Delta C + n) \neq \emptyset$ where $n=\epsilon + 2a+1$.
  One can easily check that $-I+n \subseteq \II(3, \epsilon)\setminus I$, 
  and by \ref{mainOP:eq} it follows that 
  every cycle of $\Gamma$ satisfies condition 1 of Theorem \ref{doubling3}, which 
  therefore guarantees that there is an $F$-starter whose cycle passing through 
  $\infty$ has length $h = 2m+4a+3$. 
\end{proof}

In the following we provide two improvements of Theorem \ref{mainOP} concerning the cases where all the $\ell_i$s are even or odd. 

\begin{thm}\label{mainOPeven} 
Let $F=[h, 2\ell_1, \ldots, 2\ell_r, \;^{2}\ell_{r+1}, \ldots, \;^{2}\ell_{s}]$ 
where the $\ell_i$s are all even,
$2\leq \ell_1< \ell_2< \ldots < \ell_r$, and $\ell_j\geq 4$ for every $r+1\leq j\leq s$.
Then $OP^\s(F)$ has a $(1-\s)$-pyramidal solution whenever 
\begin{equation}\label{mainOPeven:eq}
\textstyle{
h\geq 4\left(2s-\left\lfloor \frac{2}{\ell_1}\right\rfloor\right)(\ell+2)  
+4\left\lfloor \frac{2}{\ell_1}\right\rfloor - 1
\;\;\text{where}\;\;\ell = \max\{\ell_i\mid 1\leq i\leq s\}.
}
\end{equation}
In particular, the result holds when $h \geq 8s(\ell+2)$.
\end{thm}
\begin{proof} 
Let $K=\{\ell_i \mid 3\leq \ell_i, 1\leq i\leq r\}$ and $K'=\{\ell_{r+1}, \ell_{r+2}, \ldots, \ell_{s}\}$, and set $a=\left\lfloor 2/\ell_1\right\rfloor$. 
Note that $\ell_1\not\in K \Leftrightarrow \ell_1=2 \Leftrightarrow a=1$. 
Therefore we can write $F = [h, \;^a4, 2K, \;^2K']$. 

As in the proof of Theorem \ref{mainOP}, it is enough to show that 
if $h$ is an odd integer satisfying \eqref{mainOPeven:eq}, then there is an 
$F$-starter whose cycle passing through $\infty$ has length $h$.
  
  Let $L = K\ \cup\ K'$ and $m= (h-1)/2 -2a -1$ 
  (recall that $h$ is odd). Note that $|L| +a = s$, and by \eqref{mainOPeven:eq} we have that
  $m\geq  \left(4|L|+2a\right)(\ell+2) -2$. 
  Therefore the graph
  $[L, \;^a2 \mid m]_\epsilon$ has size 
  \begin{equation*}
    \epsilon = 
     \sum_{\ell\in L} \ell +  2a  + m \geq 4|L|+2a + m = 
     \left(4|L|+2a\right)(\ell+3) -2, 
  \end{equation*}
  and Theorems \ref{thm:main1:1} and \ref{thm:main2even:2} guarantee that 
  $[L, \;^a2 \mid m]_\epsilon$ has a graceful labeling $\Gamma$ 
  such that
  \begin{equation}\label{mainOPeven:eq2}
    \DeltaP\Gamma \supseteq \II(3, \epsilon)\setminus I, 
    \;\;\text{where}\;\;
    I=\II\big(3, (2|L|+a)(\ell+3) + a -1\big).
  \end{equation}
  Clearly $\Gamma$ satisfies condition 2 of Theorem \ref{doubling3}.  
  Furthermore, for every cycle $C$ of $\Gamma$ 
  we have that $-I\ \cap\ \Delta C\neq \emptyset$, hence
  $(-I+n)\ \cap\ (\Delta C + n) \neq \emptyset$ where $n=\epsilon + 2a+1$.
  One can easily check that $-I+n \subseteq \II(3, \epsilon)\setminus I$, 
  and by \ref{mainOPeven:eq2} it follows that 
  every cycle of $\Gamma$ satisfies condition 1 of Theorem \ref{doubling3}, which 
  therefore guarantees that there is an $F$-starter whose cycle passing through 
  $\infty$ has length $h = 2m+4a+3$. 
\end{proof}  

\begin{thm}\label{mainOPodd} 
Let $F=[h, 2\ell_1, \ldots, 2\ell_r, \;^{2}\ell_{r+1}, \ldots, \;^{2}\ell_{s}]$ 
where the $\ell_i$s are all odd,
$3\leq \ell_1< \ell_2< \ldots < \ell_r$, and $\ell_j\geq 3$ for every $r+1\leq j\leq s$.
Then $OP^\s(F)$ has a $(1-\s)$-pyramidal solution whenever 
\begin{equation}\label{mainOPodd:eq}
h> 7^{s-1}(12\ell+6)
\;\;\text{where}\;\;\ell = \max\{\ell_i\mid 1\leq i\leq s\}.
\end{equation}
\end{thm}
\begin{proof} The proof is very similar to the two previous ones.
Letting $K=\{\ell_1,$ $\ell_2, \ldots, \ell_r\}$ and $K'=\{\ell_{r+1}, \ell_{r+2}, \ldots, \ell_{s}\}$,
we can write $F = [h,  2K, \;^2K']$. 
It is enough to show that 
if $h$ is an odd integer satisfying \eqref{mainOPodd:eq}, then there is an 
$F$-starter whose cycle passing through $\infty$ has length $h$.
  
  Let $L = K\ \cup\ K'$ and $m= (h-3)/2$. 
  By \eqref{mainOPodd:eq} we have that the graph
  $[L, \mid m]_\epsilon$ has size 
  \begin{equation*}
    \epsilon = 
     \sum_{\ell\in L} \ell +  m \geq 7^{s-1}(6\ell+3).
  \end{equation*}
  and Theorem \ref{thm:main2odd:2} guarantees that 
  $[L, \mid m]_\epsilon$ has a graceful labeling $\Gamma$ such that
  $\DeltaP\Gamma \supseteq \II(3, \epsilon)\setminus I$, 
  where $I=\II\big(3, 7^{s-1}(2\ell+1)\big)$.
  One can esaily check that $\Gamma$ satisfies the assumptions of 
  Theorem \ref{doubling3}, 
  which therefore guarantees that there is an $F$-starter whose cycle passing through 
  $\infty$ has length $h = 2m+3$. 
\end{proof}

The following theorem concludes this section with further results on 
$OP(\lambda, F)$ for even $\lambda$, 
which are straightforward applications of Theorem \ref{1rotsols2} to the
graceful labelings built in Corollary \ref{cor:main1}, and Theorems \ref{thm:main2odd:2}
and \ref{thm:main2general:1}.
\\

\noindent
\textbf{Theorem \ref{mainOP2}.}
\emph{ 
Let $F=[h, \ell_1,\ldots, \ell_r, \ell'_1,\ldots, \ell'_s,]$ be any simple $2$-regular graph,
where the $\ell_i$s are even and the $\ell'_j$s are odd. 
Also, set $\ell=\max\{\ell_1,\ldots, \ell_r\}$ and  $\ell'=\max\{\ell'_1,\ldots, \ell'_s\}$.
Then $OP(2, F)$ has a $1$-rotational solution whenever
\[h> 
  \begin{cases}
    12r(\ell+3) + 7^{s}(2\ell'+1) -6  &\text{if $r,s > 0$}, \\    
    2(r+1)(\ell+1) +5 &\text{if $r> 0=s$},  \\    
    3\cdot7^{s-1}(2\ell'+1) &\text{if $s> 0=r$}. \\       
  \end{cases} 
\]   
}

We point out that this last result does not require that $F$ be 
a single-flip $2$-factor.


\section{Conclusion}
This paper deals with the Oberwolfach problem $OP(F)$, posed by Ringel in 1967, which asks for 
a decomposition of $K_v$ into copies of a given $2$-regular graph $F$:
an $F$-factorization of $K_v$. 
Necessarily, for such a decomposition to exist, 
the order of $F$ must be odd. When $F$ has even order, we consider the
maximum packing version $OP^-(F)$ and the minimum covering version
$OP^+(F)$, which ask for $F$-factorizations of $K_v-I$ and $K_v+I$, respectively, where $I$ 
is a $1$-factor of $K_v$. A further extension of the original problem consists of replacing
$K_v$ by $\lambda K_v$ (the $\lambda$-fold complete graph), and 
$K_v\pm I$ by $\lambda K_v \pm I$. 
However, since we do not require that a factorization have pairwise distinct factors, it is enough to focus on the original problem $OP(F)$, its two variants $OP^\pm(F)$, 
and $OP(2, F)$ which asks for an $F$-factorization of
$2K_v$. 

Although these problems have been extensively studied, except for $OP^+(F)$ which has been formally investigated only recently, they remain widely open.
When $\lambda=1$, Theorem \ref{mainOP} gives the solvability provided that $F$ 
is a single-flip $2$-regular graph with a cycle whose length is greater than a given lower bound. 
Theorems \ref{mainOPeven} and \ref{mainOPodd} improve this bound  when $F$ has additional properties. 
When $\lambda=2$, 
Theorem \ref{mainOP2} gives a similar result to Theorem \ref{mainOP} 
but without requiring that $F$ be single-flip.
Furthermore, our solutions are pyramidal since they have an automorphism group fixing at most two points and acting sharply transitively on the remaining. 

The constructions in this paper rely on a doubling construction 
defined in \cite{Buratti Traetta 12}
which can be applied to graceful and $\alpha$-labelings of zillion graphs, 
which are 
$2$-regular graphs with a vertex removed.
Zillion graphs can be  described in terms of the size of their components which are cycles and exactly one path.
In Sections 4 and 5, we show that such labelings exist for zillion graphs provided that the length of the path component satisfies
a given lower bound (Theorem \ref{mainGL}).

We also point out the analogy between the class 
of zillion graphs 
and the class of trees: both these types of graphs have size equal to the order minus 1. 
The gracefulness of trees has been widely studied, in order to settle the 
graceful tree conjecture which, despite the efforts of many, remains open. 
Surprisingly, 
despite having a structure simpler than trees, 
the gracefulness of zillion graphs is still uncertain in general.

Graceful labelings of those with two components, one cycle and one path, were built in \cite{Traetta 13} settling a conjecture posed by Frucht and Salinas  in 1985. 
As far as we know, Theorem \ref{mainGL} represents the first result on the gracefulness of zillion graphs with more than two components.
Our results led us to formulate three conjectures concerning the gracefulness of
zillion graphs
(Conjectures \ref{conjectureweak}--\ref{conjecturestrongest}).
Corollary \ref{cor:main1} actually proves Conjecture \ref{conjectureweak} when all the cycles have even length. In fact, we believe that the lower bound on the path length need not depend on the cycle lengths, leading us to propose Conjecture \ref{conjecturestrong}. 
Finally, we formulate Conjecture \ref{conjecturestrongest} concerning the size of a graceful zillion graph.

\section*{Acknowledgements}
The authors gratefully acknowledge support from the following sources.  
A.C.\ Burgess and P.\ Danziger have received support from NSERC Discovery Grants RGPIN-2019-04328 and 
RGPIN-2016-04178, respectively. 
T. Traetta has received support from GNSAGA of Istituto Nazionale di Alta Matematica.

\end{document}